\documentclass{article} % For LaTeX2e
\usepackage{iclr2026_conference,times}

\usepackage{amsmath,amsfonts,bm}

\def\1{\bm{1}}

\def\mA{{\bm{A}}}
\def\mB{{\bm{B}}}

\def\mD{{\bm{D}}}

\def\mI{{\bm{I}}}

\def\mM{{\bm{M}}}

\def\mQ{{\bm{Q}}}

\def\mX{{\bm{X}}}

\DeclareMathAlphabet{\mathsfit}{\encodingdefault}{\sfdefault}{m}{sl}
\SetMathAlphabet{\mathsfit}{bold}{\encodingdefault}{\sfdefault}{bx}{n}

\newcommand{\R}{\mathbb{R}}

\usepackage{hyperref}
\usepackage{url}

\usepackage[hyperpageref]{backref}
\hypersetup{ 
    colorlinks = true,
    urlcolor = yaleblue,
    linkcolor = yaleblue,
    citecolor = yaleblue% {blue!50!black}
}
\renewcommand*{\backref}[1]{}% for backref < 1.33 necessary
\renewcommand*{\backrefalt}[4]{%
   \ifcase #1 %
     \footnotesize{(Not cited.)}%
   \or
     \footnotesize{(Cited on page~#2)}%
   \else
     \footnotesize{(Cited on page~#2)}%
\fi }

\usepackage{graphicx}
\usepackage{apptools}
\usepackage[flushleft]{threeparttable}
\usepackage{array,booktabs,makecell}
\usepackage{multirow}
\usepackage{nicefrac}
\usepackage{amsthm}
\usepackage{amsmath,amsfonts,bm}
\usepackage{wrapfig}
\usepackage{caption}
\usepackage{subcaption}
\usepackage{siunitx}
\usepackage{thm-restate}
\usepackage{nccmath}
\usepackage{bbm}
\usepackage[shortlabels]{enumitem}
\usepackage{cleveref}

\usepackage{algorithmic}
\usepackage{algorithm}
\usepackage{color}
\usepackage{colortbl}
\definecolor{bgcolor}{rgb}{0.76,0.88,0.50}
\definecolor{bgcolor0}{rgb}{0.93,0.99,1}
\definecolor{bgcolor1}{rgb}{0.8,1,1}
\definecolor{bgcolor2}{rgb}{0.8,1,0.8}
\definecolor{bgcolor3}{rgb}{0.50,0.90,0.50}
\usepackage{tcolorbox}
\usepackage{pifont}
\definecolor{yaleblue}{rgb}{0.06, 0.3, 0.57}

\newcommand{\algname}[1]{{\small \sf #1}}

\newcommand{\squeeze}{\textstyle}

\newcommand{\norm}[1]{\left\| #1 \right\|}

\newcommand{\inp}[2]{\left\langle#1,#2\right\rangle} % inner product
\newcommand{\Exp}[1]{{\rm \mathbb{E}}\left[#1\right]}
\newcommand{\ExpSub}[2]{{\rm \mathbb{E}}_{#1}\left[#2\right]}

\newcommand{\parens}[1]{\left( #1 \right)}
\newcommand{\brac}[1]{\left\{ #1 \right\}}

\DeclareMathOperator{\range}{range}     % range
\newcommand{\cO}{\mathcal{O}}

\newcommand{\cS}{\mathcal{S}}

\newcommand{\cX}{\mathcal{X}}

\newcommand{\ct}{\eta}
\newcommand{\blambda}{\boldsymbol{\Lambda}}

\newcommand{\eqdef}{:=} 

\makeatletter
\newcommand{\vast}{\bBigg@{4}}

\newtheorem{fact}{Fact}
\theoremstyle{plain}
\newtheorem{theorem}{Theorem}[section]

\newtheorem{lemma}[theorem]{Lemma}

\theoremstyle{definition}
\newtheorem{definition}[theorem]{Definition}
\newtheorem{assumption}[theorem]{Assumption}
\theoremstyle{remark}
\newtheorem{remark}[theorem]{Remark}
\newtheorem{example}[theorem]{Example}

\newcommand{\prox}{\textnormal{prox}}
\allowdisplaybreaks

\title{Tighter Performance Theory of FedExProx}

\author{Wojciech Anyszka\thanks{The work of Wojciech Anyszka was conducted during a VSRP internship at KAUST.} \,$^1$,
Kaja Gruntkowska$^2$,
Alexander Tyurin$^{3,4}$ \&
Peter Richt{\'a}rik$^2$ \\
$^1$ Princeton University, Princeton, United States \\
$^2$ King Abdullah University of Science and Technology, Thuwal, Saudi Arabia \\
$^3$ AXXX, Moscow, Russia \\
$^4$ Applied AI Institute, Moscow, Russia
}

\iclrfinalcopy % Uncomment for camera-ready version, but NOT for submission.
\begin{document}

\maketitle

\begin{abstract}
    We revisit \algname{FedExProx}--a recently proposed distributed optimization method designed to enhance convergence properties of parallel proximal algorithms via extrapolation. In the process, we uncover a surprising flaw: its known theoretical guarantees on quadratic optimization tasks are no better than those offered by the vanilla Gradient Descent (\algname{GD}) method. Motivated by this observation, we develop a novel analysis framework, establishing a tighter linear convergence rate for non-strongly convex quadratic problems. By incorporating both computation and communication costs, we demonstrate that \algname{FedExProx} can indeed provably outperform \algname{GD}, in stark contrast to the original analysis. Furthermore, we consider partial participation scenarios and analyze two adaptive extrapolation strategies--based on gradient diversity and Polyak stepsizes---again significantly outperforming previous results. Moving beyond quadratics, we extend the applicability of our analysis to general functions satisfying the Polyak-Łojasiewicz condition, outperforming the previous strongly convex analysis while operating under weaker assumptions. Backed by empirical results, our findings point to a new and stronger potential of \algname{FedExProx}, paving the way for further exploration of the benefits of extrapolation in federated learning.
\end{abstract}

\section{Introduction}\label{sec:intro}
Federated Learning (FL) \citep{konevcny2016federated, mcmahan2017communication} is a distributed machine learning paradigm where multiple edge devices collaboratively train a global model without the need for centralized data collection. Instead of transmitting raw data, each client computes updates based on their local datasets, and these are then aggregated at a central server to update the global model. This approach ensures that no sensitive information is shared with the server or other clients, making it particularly useful in privacy-sensitive domains, such as healthcare \citep{rieke2020future} and recommendation systems \citep{hard2018federated}.
Formally, FL solves the optimization problem
\begin{align}
    \label{eq:main_problem}
    \squeeze \min\limits_{x \in \R^d}\left\{f(x) \eqdef \frac{1}{n} \sum\limits_{i=1}^n f_i(x)\right\},
\end{align}
where $f_i: \R^d \to \R$ represents the empirical risk associated with the local dataset stored on client~$i \in [n]$.
Despite its advantages, FL introduces several nontrivial challenges, including \textit{communication delays} caused by limited network bandwidth, and \textit{partial participation} of clients due to random outages \citep{kairouz2021advances}. Addressing these issues is essential to making FL efficient and reliable in real-world applications.

\subsection{Related Work}\label{sec:lit_review}
Let us briefly review several existing methods used to solve problem \eqref{eq:main_problem} in FL scenarios. For a detailed overview of the notation used, please refer to Table~\ref{table:notation}.

\textbf{Gradient descent.} The simplest and most naive approach is the Gradient Descent (\algname{GD}) method \citep{nesterov2018lectures}, iterating 
$x_{k+1} = x_k - \gamma \nabla f(x_k) = x_k - \gamma \frac{1}{n} \sum_{i=1}^{n} \nabla f_i(x_k),$
where $x_0$ is a starting point and $\gamma>0$ is the stepsize. 
Let us introduce the key assumptions typically employed in the analysis of \algname{GD} in the convex world.

\begin{assumption}
    \label{ass:convex}
    The functions $f_i$ are proper, closed and convex for all $i \in [n]$, and the function $f$ attains a minimum at some (potentially non-unique) point $x_* \in \R^d.$  
\end{assumption}
\begin{assumption}
    \label{ass:local_lipschitz_constant}
    The function $f$ is $L$--smooth, i.e., $\norm{\nabla f(x) - \nabla f(y)} \leq L \norm{x - y}$ $\forall x, y \in \R^d.$ Additionally, each $f_i$, $i \in [n],$ is differentaible and $L_i$--smooth, with $L_{\max} \eqdef \max_{i \in [n]} L_i.$
    \end{assumption}
Under these assumptions, classical optimization theory guarantees that \algname{GD} solves problem~\eqref{eq:main_problem} and finds $\bar{x} \in \R^d$ such that $f(\bar{x}) - f(x_*) \leq \varepsilon$ in~$\cO(\nicefrac{L R^2}{\varepsilon})$ iterations, where~$R^2 \eqdef \norm{x_* - x_0}^2.$ 
In an FL setting, finding a solution to the main problem requires each client to participate in every communication round, and thus to communicate with the server
\begin{align}
    \label{eq:gd}
    \squeeze \cO\left(\frac{L R^2}{\varepsilon}\right)
\end{align}
times, transmitting the local gradient $\nabla f_i(x_k)$ at each iteration $k$. While this approach is conceptually simple, its direct application in distributed environments leads to several challenges related to communication overhead and scalability.

\textbf{Partial participation.}
In real-world FL scenarios, not all clients remain active for the entire duration of training. Instead, a subset of clients is chosen at each iteration, typically based on practical factors such as device availability (e.g., battery life or network conditions) or statistical considerations (e.g., data heterogeneity) \citep{kairouz2021advances, tyurin2022computation}.

\textbf{Communication bottleneck.}
Another key problem in distributed training is the \textit{communication bottleneck} \citep{ramesh2021zero, kairouz2021advances}. Since the overall performance of a distributed algorithm is the product of the number of communication rounds $K$ needed to find a solution and the cost $C$ of one such round, there exist two main strategies to addressing this issue: $(i)$~minimizing~$C$ and/or $(ii)$~minimizing~$K$.
Objective $(i)$~is typically achieved via compression of the information transmitted between the clients and the server \citep{beznosikov2020biased, gruntkowska2024improving}. Objective $(ii)$~can be accomplished by increasing the amount of local computation performed by clients between rounds, allowing for less frequent communication.

\textbf{FedAvg and modern local methods.}
A common strategy for ensuring communication efficiency that falls into category $(ii)$ is \textit{local training}. The idea is simple: allow clients to do more work before transmitting the results to the server. If executed with care, this approach leads to each communication round providing more ``informative'' updates, ultimately resulting in more effective changes to the global model.
One of the most popular methods in this class is Federated Averaging (\algname{FedAvg}), introduced by \citet{mcmahan2017communication}.
Within this framework, a subset of clients is selected in each round to perform local training using a gradient-based method of their choice (e.g., \algname{GD}). This is followed by an averaging step on the server, where the local updates are aggregated.
Despite its popularity, the theoretical properties of \algname{FedAvg} are somewhat limited \citep{khaled2019first,koloskova2020unified}. A major issue is \textit{client drift}, which can lead to slower convergence \citep{karimireddy2020scaffold}. This problem can be mitigated using modern drift correction techniques \citep{karimireddy2020scaffold,gorbunov2021local,mishchenko2022proxskip}. One notable advancement in this area is \algname{FedExP}, introduced by \citet{jhunjhunwala2023fedexp}, which extends \algname{FedAvg} by leveraging extrapolation to better control server updates.

\textbf{FedProx and proximal methods.}
All methods discussed so far assume that clients have access to the gradients $\{\nabla f_i\}$ of local functions. The class of \textit{proximal methods} \citep{bertsekas2011incremental, ryu2014stochastic, khaled2022faster, richtarik2024unified} relies on a different, more powerful oracle--the \emph{proximal operator} $\prox_{\gamma f} : \R^d \rightarrow \R^d$, defined by
\begin{align}
    \label{eq:prox}
    \squeeze \prox_{\gamma f} (x) = \arg\min\limits_{z \in \R^d} \left\{f(z) + \frac{1}{2 \gamma} \norm{z - x}^2\right\}.
\end{align}
Here, $f$ is a convex function and $\gamma>0$ plays a rule of a regularization parameter.
It is important to note that computing $\prox_{\gamma f} (\cdot)$ is an optimization problem on its own. However, the analyses of such methods typically do not consider how this subproblem should be solved, and instead assume that computing $\prox_{\gamma f} (\cdot)$ is computationally cheap and can be performed using any iterative solver, e.g., \algname{GD}, \algname{Newton's method} or \algname{LBFGS} \citep{nesterov2018lectures,liu1989limited}.

There exist numerous proximal methods addressing the single-node case. A classic approach is the Proximal Point Method (\algname{PPM}) by \citet{rockafellar1976monotone}, which iterates
$x_{k+1}=\prox_{\gamma_k f} (x_k)$. Stochastic versions of \algname{PPM} have also been studied \citep{ryu2014stochastic, bianchi2016ergodic, patrascu2018nonasymptotic}. Recently, interest in stochastic proximal methods has surged, with promising results from several works \citep{condat2022RandProx, traore2024variance, sadiev2024stochastic, combettes2024randomly}.
In the context of distributed learning, one relevant work by \citet{li2020federated} builds on the principles of \algname{FedAvg} to introduce its proximal variant--\algname{FedProx}. The method replaces local gradient-based updates with proximal operator computations, leading to the update rule
\begin{align}
    \label{eq:fedprox}
    \squeeze x_{k+1} = \frac{1}{n} \sum\limits_{i=1}^n \prox_{\gamma f_i} (x_k).
\end{align}
By leveraging proximal operators, \algname{FedProx} addresses some of the limitations of traditional \algname{GD} in FL. However, simply averaging local updates at the server remains suboptimal and can be improved.

\textbf{FedExProx.}
A recent advancement, and the main focus of this paper, is \algname{FedExProx}, introduced by \citet{li2024power}. This method improves on \algname{FedProx} by employing the technique of \textit{extrapolation}\footnote{Notably, the idea of extrapolation was explored by \citet{jhunjhunwala2023fedexp} in developing \algname{FedExP}.} \citep{combettes1997convex,necoara2019randomized}. The algorithm operates under the \emph{interpolation regime}:
\begin{assumption}
    \label{ass:inter}
    There exists $x_* \in \R^d$ such that $\nabla f_i(x_*) = 0$ for all $i \in [n].$ We denote the set of such minimizers by $\cX_* \eqdef \{x\in\R^d: \nabla f_i(x) = 0\}$ and the projection of $x$ onto $\cX_*$ by $\Pi(x)$ .
\end{assumption}
This assumption is relatively mild and is often met in practice, e.g., for over-parameterized models \citep{ma2018power}.
Instead of \eqref{eq:fedprox}, \algname{FedExProx} introduces the extrapolated update
\begin{align}
    \label{eq:fedexprox}
    \squeeze x_{k+1} = x_{k} + \alpha_k \left(\frac{1}{n} \sum\limits_{i=1}^n \prox_{\gamma f_i} (x_k) - x_{k}\right),
\end{align}
where $\alpha_k$ serves as the \emph{extrapolation parameter}. The algorithm is analyzed by reformulating step~\eqref{eq:fedexprox} in terms of \textit{Moreau envelopes} \citep{moreau1965}: the update rule can be expressed equivalently as
\begin{align}\label{eq:fedexprox_moreau}
    \squeeze x_{k+1} = x_{k} - \alpha_k \gamma \frac{1}{n} \sum\limits_{i=1}^n \nabla M_{f_i}^{\gamma}(x_{k}),
\end{align}
where the Moreau envelope $M_f^{\gamma} \,:\, \R^d \to \R$ of a convex function $f$ is defined as
\begin{align}\label{eq:moreau}
    \squeeze M_f^{\gamma}(x) \eqdef \min\limits_{z \in \R^d} \left\{f(z) + \frac{1}{2 \gamma} \norm{z - x}^2\right\} \quad \forall x \in \R^d.
\end{align}
Notably, when $\alpha_k=1$, \eqref{eq:fedexprox} reduces to \eqref{eq:fedprox}. However, this choice is not optimal, and choosing $\alpha_k > 1$ yields better complexity.
If $\alpha_k$ is set to be constant across iterations of the algorithm, then using the optimal value $\alpha_k \equiv \alpha = \frac{1}{\gamma L_{\gamma}} >1$ guarantees that \algname{FedExProx} converges after 
\begin{align}
    \label{eq:fedexproxcompl}
    \squeeze \cO\left(\frac{L_{\gamma} (1 + \gamma L_{\max}) R^2}{\varepsilon}\right)
\end{align}
communication rounds in the convex case, where $L_{\gamma}$ is the smoothness constant of the function $M^{\gamma}(x) \eqdef \frac{1}{n}\sum_{i=1}^n M_{f_i}^{\gamma}(x)$. \citet{li2024power} show that \eqref{eq:fedexproxcompl} is not larger than $\cO\left(\nicefrac{L_{\max} R^2}{\varepsilon}\right)$, and can be significantly smaller in practical scenarios. Thus, \algname{FedExProx} offers provably better iteration complexity compared to both \algname{FedProx} and \algname{FedExP} \citep{li2020federated, jhunjhunwala2023fedexp}.

\section{Contributions}

\textbf{1.} Although the iteration complexity \eqref{eq:fedexproxcompl} of \algname{FedExProx} outperforms both \algname{FedProx} and \algname{FedExP}, we find that it is \emph{not superior} to that of the simplest baseline, \algname{GD}, on quadratic optimization tasks (see Remark~\ref{rem:pessimistic}). This raises a critical question regarding the utility of employing the proximal oracle in the first place.

\textbf{2.} Motivated by these pessimistic findings, we revisit the analysis of \algname{FedProx} and prove a significantly more optimistic iteration complexity result. We establish a \emph{linear} convergence rate for solving non-strongly convex quadratic distributed optimization problems of the form \eqref{eq:main_problem} (see Theorem \ref{thm:quad}) and demonstrate that our new result does indeed lead to the conclusion that \algname{FedExProx} can significantly outperform \algname{GD}. Our analysis assumes a realistic model that accounts for both computation and communication times--critical factors in real-world distributed optimization. We show that the total \textit{time complexity} of \algname{FedExProx} is never worse than that of \algname{GD}, and can be strictly better when communication time dominates computation time (see Theorem~\ref{thm:fedexprox_quad_time}), which is typically the case in FL scenarios. This stands in stark contrast to the previous analysis by \citet{li2024power}, underscoring the superiority of our new findings.

\textbf{3.} To account for the stochastic setting that involves partial participation of clients, we complement the above results with client sampling (see Theorems \ref{thm:quad_iter_stoch}, \ref{thm:fedexprox_quad_time_stoch}, \ref{thm:fedexprox_pl_iter_stoch} and \ref{thm:quad_iter_ada_stoch}).

\textbf{4.} Beyond constant extrapolation, we establish two novel results (Theorems \ref{thm:quad_iter_ada} and \ref{thm:quad_iter_ada_stoch}) that incorporate smoothness-adaptive strategies, based on gradient diversity (\algname{FedExProx-GraDS}), and Polyak stepsize (\algname{FedExProx-StoPS}), again significantly improving upon the result previously established by \citet{li2024power} in the quadratic case.

\textbf{5.} We extend the analysis beyond quadratics to \emph{arbitrary convex functions} satisfying the Polyak-{\L}ojasiewicz (P{\L}) condition, obtaining a linear convergence rate (see Theorems \ref{thm:fedexprox_pl_iter} and \ref{thm:fedexprox_pl_iter_stoch}). In contrast, \citet{li2024power} derived a linear rate under strong convexity for the function $f$. Our approach is not only more general, as it relies on weaker assumptions, but also demonstrates improved dependence on problem-specific constants. Additionally, we establish a result in the P{\L} setting that accounts for inexact computations of the proximal mappings (Theorem \ref{thm:fedexprox_pl_inexact}).

\textbf{6.} The theoretical findings are validated with empirical experiments, which demonstrate the robustness and applicability of our framework.

\section{Not Better than \algname{GD} on Quadratics}
\label{sec:not_better}

Let us take a closer look at the complexity result \eqref{eq:fedexproxcompl}.
% This complexity depends on the regularization parameter~$\gamma$ introduced in definition~\eqref{eq:prox}.
Suppose for now that we solve~\eqref{eq:prox} using \algname{GD}. Then, the number of iterations needed to find $\prox_{\gamma f_i} (x)$ with accuracy $\varepsilon$ is
\begin{align*}
    \squeeze \cO\left(\frac{L_i + \nicefrac{1}{\gamma}}{\nicefrac{1}{\gamma}} \log \frac{1}{\varepsilon}\right) = \cO\left((\gamma L_i + 1) \log \frac{1}{\varepsilon}\right), 
\end{align*}
since the function $f_i(z) + \frac{1}{2 \gamma} \norm{z - x}^2$ is $(L_i + \nicefrac{1}{\gamma})$--smooth and $\nicefrac{1}{\gamma}$--strongly convex. Hence,~$\gamma$ controls the difficulty of calculating the proximal operator: the larger it is, the more difficult the problem. Therefore, to accelerate local computations, one would prefer to choose $\gamma$ as small as possible.
We can formalize this intuition by letting $\pi(\gamma)$ be the time per one iteration of \algname{FedExProx} and making the following reasonable assumption:
\begin{assumption}
    \label{ass:local_compl}
    The time complexity $\pi(\gamma)$ of a \algname{FedExProx} step is a non-decreasing function of~$\gamma$.
\end{assumption}
Combining it with the iteration complexity result \eqref{eq:fedexproxcompl} established by \citet{li2024power}, the total time required by \algname{FedExProx} with $\alpha_k \equiv \nicefrac{1}{\gamma L_{\gamma}}$ to find $\bar{x}$ such that $\mathbb{E}[f(\bar{x})] - f(x_*) \leq \varepsilon$ is
\begin{align*}
    \squeeze T(\gamma) \eqdef \pi(\gamma) \times \frac{L_{\gamma} (1 + \gamma L_{\max}) R^2}{\varepsilon}.
\end{align*}
With this definition in place, we can now translate the main result of \citet{li2024power} (Theorem 1) to quadratic optimization problems, yielding the following pessimistic outcome:
\begin{restatable}{theorem}{THMPESSIMISTIC}\label{thm:pessimistic}
    Let Assumptions~\ref{ass:inter} and~\ref{ass:local_compl} hold. Consider solving a non-strongly convex quadratic optimization problem of the form \eqref{eq:main_problem}, where
        $f_i(x) = \frac{1}{2} x^\top \mA_i x - b_i^\top x$
    for all $i \in [n],$ with $\mA_i \in \textnormal{Sym}^{d}_{+}$ and $b_i \in \R^d$.
    Then the assumptions of Theorem 1 by \citet{li2024power} hold, and
    \begin{align}
        \label{eq:pess_inequality}
        \squeeze T(\gamma) = \pi(\gamma) \times \frac{L_{\gamma} (1 + \gamma L_{\max}) R^2}{\varepsilon} \geq \pi(0) \times \frac{L R^2}{\varepsilon}
    \end{align}
    for all $\gamma > 0.$ Moreover, when $\gamma \to 0$, then $\pi(\gamma) \times \frac{L_{\gamma} (1 + \gamma L_{\max}) R^2}{\varepsilon} \to \pi(0) \times \frac{L R^2}{\varepsilon}$, and \algname{FedExProx} effectively reduces to \algname{GD}.
\end{restatable}

\begin{remark}\label{rem:pessimistic}
    In light of Theorem~\ref{thm:pessimistic}, \algname{GD} performs no worse than \algname{FedExProx} by \citet{li2024power}, even if the time complexity $\pi(\gamma)$ of one \algname{FedExProx} step does not depend on $\gamma$. Indeed, it holds that 
    $$\squeeze \frac{L_{\gamma} (1 + \gamma L_{\max}) R^2}{\varepsilon} \geq \frac{LR^2}{\varepsilon}$$ for all $\gamma > 0$ (see the proof of Theorem \ref{thm:pessimistic}).
\end{remark}

These results, based on the original analysis of \algname{FedExProx}, lead us to a rather disappointing conclusion: from a theoretical time complexity perspective, the optimal strategy appears to be using vanilla \algname{GD}, disregarding proximal oracles entirely. The question that remains is: is this an inherent limitation of the method itself or merely a consequence of suboptimal analysis by \citet{li2024power}?

\section{Tighter and More Optimistic Results}
\label{sec:tight}

% Are we truly constrained to \algname{GD}, or is there room for improvement?
Our refined analysis of \algname{FedExProx} demonstrates that a tighter iteration complexity bound \emph{can} be derived, leading to significantly more optimistic results. In the following sections, we present a novel analysis that provides an improved complexity result for \algname{FedExProx} and sheds light on its true capabilities. To ensure clarity and build intuition, we begin with quadratic optimization tasks, where all relevant quantities can be explicitly derived. In Section \ref{sec:pl}, we extend these advancements to \emph{general convex functions} satisfying the P{\L} condition.
% It turns out that a tighter iteration complexity bound for \algname{FedExProx} can be derived, revealing significantly more optimistic results.
% The following theorem presents a refined analysis, providing a new, improved complexity result for \algname{FedExProx} and shedding light on its true capabilities.

\begin{restatable}{theorem}{FEDEXPROXQITERTIGHT}\label{thm:quad}
    Fix any $\gamma > 0$ and consider solving non-strongly convex quadratic optimization problem \eqref{eq:main_problem} where $f_i(x) = \frac{1}{2} x^\top \mA_i x - b_i^\top x$
    for all $i \in [n],$ with $\mA_i \in \textnormal{Sym}^{d}_{+}$ and $b_i \in \R^d$. Under Assumption \ref{ass:inter}, \algname{FedExProx} with $\alpha = \nicefrac{1}{\gamma L_{\gamma}}$ finds $\bar{x}$ such that $\mathbb{E}[f(\bar{x})] - f(x_*) \leq \varepsilon$ after
    \begin{align}
        \label{eq:tighter_quad}
        \squeeze \cO\left(\frac{L_{\gamma}}{\mu^{+}_{\gamma}} \log \frac{1}{\varepsilon}\right)
    \end{align}
    iterations, where $L_{\gamma}$ is a smoothness constant of $M^{\gamma}$ and $\mu^{+}_{\gamma}$ is the smallest non-zero eigenvalue of the matrix $\nabla^2 M^{\gamma}.$
\end{restatable}
\begin{remark}
    To the best of our knowledge, under the assumptions of Theorem~\ref{thm:quad}, \algname{GD} requires 
    \begin{align}
        \label{eq:tighter_quad_gd}
        \squeeze \cO\left(\frac{L}{\mu^{+}} \log \frac{1}{\varepsilon}\right)
    \end{align}
    iterations to solve the quadratic optimization problem, where $\mu^{+}$ is the smallest non-zero eigenvalue of the matrix $\mA = \frac{1}{n} \sum_{i=1}^{n} \mA_i$ \citep{richtarik2020stochastic}.
\end{remark}
Let us now demonstrate that our new result does indeed provide a tighter bound, leading to the conclusion that \algname{FedExProx} can in fact outperform \algname{GD}.

\subsection{One step time complexity $\pi(\gamma)$ is $\ct + \tau \left(\gamma L_{\max} + 1\right)$}\label{sec:time_cond}

As introduced in Section \ref{sec:intro}, the time per one global iteration of \algname{FedExProx} has two main sources:

1. \textbf{Local computation:} In large-scale problems, each step~\eqref{eq:fedexprox} requires clients to compute $\prox_{\gamma f_i}(x_k)$ iteratively. One of the simplest solvers is \algname{GD}, which returns a solution of subproblem $i$ after $\widetilde{\cO}\left(\gamma L_i + 1\right)$ local iterations\footnote{Alternatively, an accelerated method could be employed, reducing the iteration count to $\widetilde{\cO}\left(\sqrt{\gamma L_i + 1}\right)$. Our analysis can be extended to accommodate these faster solvers as well.} (see Section \ref{sec:not_better}). If each gradient calculation
takes $\tau$ seconds, the total time required for all clients to calculate $\prox_{\gamma f_i}(x_k)$ is $\widetilde{\cO}\left(\tau \times (\gamma L_{\max} + 1)\right)$ since the process is gated by the ``slowest'' client, associated with the problem with the highest smoothness constant.

2. \textbf{Communication:} Once the local computations are completed, clients must communicate their results before the server can execute the global step \eqref{eq:fedexprox}. We assume this communication takes $\ct$ seconds, which can be huge in FL environments \citep{kairouz2021advances}.

Consequently, the total time per global iteration, $\pi(\gamma)$, is proportional to $\ct + \tau \left(\gamma L_{\max} + 1\right)$ (thus satisfying Assumption~\ref{ass:local_compl}), and the total time complexity of \algname{FedExProx} is
\begin{align}\label{eq:fedexprox_quad_time}
    \squeeze T_{\ct}(\gamma) \eqdef \widetilde{\cO} \left(\left(\ct + \tau \left(\gamma L_{\max} + 1\right)\right) \times \frac{L_{\gamma}}{\mu^{+}_{\gamma}}\right).
\end{align}
Note that the total time complexity of \algname{GD} is
\begin{align*}
    \squeeze T_{\algname{GD}} \eqdef T_{\ct}(0) = \widetilde{\cO} \left(\left(\ct + \tau\right) \times \frac{L}{\mu^{+}}\right).
\end{align*}
Our next goal is to determine the value of $\gamma$ that minimizes $T_{\ct}(\gamma)$, with the hope that this time the optimal choice does not result in $\gamma\to 0$. This is indeed the case whenever $\ct > \tau$.
\begin{restatable}{theorem}{FEDEXPROXQTIME}\label{thm:fedexprox_quad_time}
    Consider the non-strongly convex quadratic optimization problem from Theorem~\ref{thm:quad}. Up to a constant factor, the time complexity \eqref{eq:fedexprox_quad_time} is minimized by some
    \begin{align*}
        \squeeze \gamma \in \left[\frac{1}{\max\limits_{i \in [n]}\lambda_{\max}(\mA_{i})}, \min\brac{\frac{\frac{\ct}{\tau}-1}{\max\limits_{i \in [n]}\lambda_{\max}(\mA_{i})}, \frac{1}{\min\limits_{i\in[n]}\lambda_{\min}^+(\mA_i)}}\right]
    \end{align*}
    % $\gamma \in \left[\frac{1}{\max_{i \in [n]}\lambda_{\max}(\mA_{i})}, \min\brac{\frac{\frac{\ct}{\tau}-1}{\max_{i \in [n]}\lambda_{\max}(\mA_{i})}, \frac{1}{\min_{i\in[n]}\lambda_{\min}^+(\mA_i)}}\right]$
    if $\frac{\ct}{\tau}\geq2$ and by some
    \begin{align*}
        \squeeze \gamma \in \left[0, \max\brac{0, \min\brac{\frac{\frac{\ct}{\tau}-1}{\max\limits_{i \in [n]}\lambda_{\max}(\mA_{i})}, \frac{1}{\min\limits_{i\in[n]}\lambda_{\min}^+(\mA_i)}}}\right]
    \end{align*}
    % $\gamma \in \left[0, \max\brac{0, \min\brac{\frac{\frac{\ct}{\tau}-1}{\max_{i \in [n]}\lambda_{\max}(\mA_{i})}, \frac{1}{\min_{i\in[n]}\lambda_{\min}^+(\mA_i)}}}\right]$
    if $\frac{\ct}{\tau}<2$. Moreover, $T_{\ct}(\gamma) \leq T_{\algname{GD}}.$
\end{restatable}
The theorem above establishes that the total time complexity of \algname{FedExProx} is no worse than that of \algname{GD}, and can be strictly better when communication time dominates computation time, highlighting the superiority of our new analysis. Notably, this was not the case for the previous analysis by \citet{li2024power} (Theorem~\ref{thm:pessimistic}). 
The following example, where $T_{\ct}(\gamma) \ll T_{\algname{GD}}$ for some $\gamma > 0$, highlights how our complexity can significantly outperform that of \algname{GD}.
% The potential for our complexity to significantly outperform that of \algname{GD} is illustrated in the following simple example, where $T_{\ct}(\gamma) \ll T_{\algname{GD}}$ for some $\gamma > 0.$
\begin{example}
    Let the matrices $\{\mA_i\}$ be diagonal, i.e., $\mA_i = \textnormal{diag}\parens{a_{i1}, \ldots, a_{id}}$ for all $i \in [n],$ and $a_{ij} > 0$ for all $i \in [n], j \in [d].$ Then, according to \eqref{eq:lmax_eig} and \eqref{eq:lmin+_eig}, we get
    \begin{align*}
        \squeeze \frac{L_{\gamma}}{\mu^{+}_{\gamma}} = \frac{\lambda_{\max}(\gamma)}{\lambda_{\min}^+(\gamma)} = \frac{\max\limits_{j \in [d]} \sum\limits_{i=1}^{n} \frac{\gamma a_{ij}}{1 + \gamma a_{ij}}}{\min\limits_{j \in [d]} \sum\limits_{i=1}^{n} \frac{\gamma a_{ij}}{1 + \gamma a_{ij}}},
    \end{align*}
    where $\lambda_{\max}(\gamma)$ and $\lambda_{\min}^+(\gamma)$ are the largest and the smallest non-zero eigenvalues of the matrix $\mM \eqdef \frac{1}{n}\sum_{i=1}^{n}\frac{1}{\gamma}(\mI - (\gamma \mA_{i} + \mI)^{-1})$, respectively (see Section \ref{sec:fedexprox_quad}).
    Taking $\gamma = 1 / \min_{i \in [n], j \in [d]}a_{ij}$ and using the fact that $1 \geq \frac{x}{1 + x} \geq \frac{1}{2}$ for all $x \geq 1,$ we can conclude that $\frac{L_{\gamma}}{\mu^{+}_{\gamma}} \leq 2.$ Thus,
    \begin{align*}
        \squeeze T_{\ct}(\gamma) = \widetilde{\cO} \left(\ct + \tau \frac{\max\limits_{i \in [n], j \in [d]}a_{ij}}{\min\limits_{i \in [n], j \in [d]}a_{ij}}\right).
    \end{align*}
    Suppose that communication is slow. Formally, let
    $\ct \geq \tau \max_{i \in [n], j \in [d]}a_{ij} / \min_{i \in [n], j \in [d]}a_{ij}.$ Then 
    $T_{\ct}(\gamma) = \widetilde{\cO} (\ct)$, while the total time complexity of \algname{GD} is 
    \begin{align*}
        \squeeze T_{\algname{GD}} = \widetilde{\Omega} \left(\ct \times \frac{L}{\mu^{+}}\right) = \widetilde{\Omega} \left(\ct \times \frac{\max\limits_{j \in [d]} \sum\limits_{i=1}^{n} a_{ij}}{\min\limits_{j \in [d]} \sum\limits_{i=1}^{n} a_{ij}}\right),
    \end{align*}
    which is at least $\max_{j \in [d]} \sum_{i=1}^{n} a_{ij} / \min_{j \in [d]} \sum_{i=1}^{n} a_{ij}$ times worse!
\end{example}
A similar improvement can be observed in the general case. However, the derivation is significantly more complex, so to maintain clarity, the example focuses on diagonal matrices only.

\section{Partial Participation}\label{sec:pp}

Thus far, we have concentrated on the full participation scenario. However, as outlined in Section~\ref{sec:lit_review}, practical FL settings often involve only a subset of clients participating in each training round. To address this, we supplement our theory with a convergence result in the stochastic setting. For illustration, we consider nice sampling (see Section~\ref{sec:nice_sampling}), where at each iteration, a subset~$\cS_k \subseteq [n]$ of clients is selected uniformly at random from all subsets of size~$S$. Although we use this sampling strategy as an example, other client selection methods can also be employed.
In this context, we can formulate the following complexity result:

\begin{theorem}
    \label{thm:quad_iter_stoch}
    Fix any $\gamma > 0$ and consider solving non-strongly convex quadratic optimization problem~\eqref{eq:main_problem}, where $f_i(x) = \frac{1}{2} x^\top \mA_i x - b_i^\top x$
    for all $i \in [n],$ with $\mA_i \in \textnormal{Sym}^{d}_{+}$ and $b_i \in \R^d$. Under Assumption \ref{ass:inter}, \algname{FedExProx} with nice sampling (Algorithm \ref{algorithm:batch_fedexprox}) with $\alpha = \nicefrac{1}{\gamma L_{\gamma, S}}$ finds $\bar{x}$ such that $\mathbb{E}[f(\bar{x})] - f(x_*) \leq \varepsilon$ after
    \begin{align}\label{eq:tighter_quad_stoch}
        \squeeze \cO\left(\frac{L_{\gamma,S}}{\mu^{+}_{\gamma}} \log \frac{1}{\varepsilon}\right)
    \end{align}
    iterations, where $L_{\gamma, S} \eqdef \frac{n-S}{S(n-1)} \frac{L_{\max}}{1+\gamma L_{\max}} + \frac{n(S-1)}{S(n-1)} L_{\gamma}$, $L_{\gamma}$ is the smoothness constant of $M^{\gamma}$ and~$\mu^{+}_{\gamma}$ is the smallest non-zero eigenvalue of the matrix $\nabla^2 M^{\gamma}.$
\end{theorem}

Assuming the same time complexity model as in Section \ref{sec:time_cond}, the total time complexity of \algname{FedExProx} with nice sampling is given by
\begin{align}\label{eq:fedexprox_quad_time_stoch}
    \squeeze T_{\ct}(\gamma, S) \eqdef \widetilde{\cO} \left(\left(\ct + \tau \left(\gamma L_{\max} + 1\right)\right) \times \frac{L_{\gamma, S}}{\mu^{+}_{\gamma}}\right).
\end{align}
As it turns out, the optimal stepsize for the stochastic setting aligns with the one used in the deterministic scenario (see Theorem \ref{thm:fedexprox_quad_time}).
\begin{restatable}{theorem}{FEDEXPROXQTIMESTOCH}\label{thm:fedexprox_quad_time_stoch}
    Up to a constant factor, the time complexity~\eqref{eq:fedexprox_quad_time_stoch} is minimized by
    \begin{align*}
        \squeeze \gamma \in \left[\frac{1}{\max\limits_{i \in [n]}\lambda_{\max}(\mA_{i})}, \min\brac{\frac{\frac{\ct}{\tau}-1}{\max\limits_{i \in [n]}\lambda_{\max}(\mA_{i})}, \frac{1}{\min\limits_{i\in[n]}\lambda_{\min}^+(\mA_i)}}\right]
    \end{align*}
    % $\gamma \in \left[\frac{1}{\max_{i \in [n]}\lambda_{\max}(\mA_{i})}, \min\brac{\frac{\frac{\ct}{\tau}-1}{\max_{i \in [n]}\lambda_{\max}(\mA_{i})}, \frac{1}{\min_{i\in[n]}\lambda_{\min}^+(\mA_i)}}\right]$
    if $\frac{\ct}{\tau}\geq2$ and by
    \begin{align*}
        \squeeze \gamma \in \left[0, \max\brac{0, \min\brac{\frac{\frac{\ct}{\tau}-1}{\max\limits_{i \in [n]}\lambda_{\max}(\mA_{i})}, \frac{1}{\min\limits_{i\in[n]}\lambda_{\min}^+(\mA_i)}}}\right]
    \end{align*}
    % $\gamma \in \left[0, \max\brac{0, \min\brac{\frac{\frac{\ct}{\tau}-1}{\max_{i \in [n]}\lambda_{\max}(\mA_{i})}, \frac{1}{\min_{i\in[n]}\lambda_{\min}^+(\mA_i)}}}\right]$
    if $\frac{\ct}{\tau}<2$.
\end{restatable}
Consequently, the conclusions from the previous section apply equally to the partial participation scenario: \algname{FedExProx} performs at least as well as \algname{GD} and can be strictly better when communication time exceeds computation time.

\section{Adaptivity}

We now turn to adaptive extrapolation strategies: \textit{gradient diversity} (\algname{GraDS}) and a variant of the classical \textit{Polyak stepsize} (\algname{StoPS}). Both were first introduced by \citet{horvath2022adaptive} and later adapted for proximal methods by \citet{li2024power}. As in the case of constant extrapolation, we refine the analysis by \citet{li2024power} for quadratic problems.
Although Theorem~\ref{thm:quad_iter_ada} refers to the full participation setting, the approach can be extended to the stochastic case (see Section~\ref{sec:ada_stoch}).

\begin{restatable}{theorem}{FEDEXPROXQITERADA}\label{thm:quad_iter_ada}
    Fix any $\gamma > 0$ and consider solving non-strongly convex quadratic optimization problem~\eqref{eq:main_problem}, where $f_i(x) = \frac{1}{2} x^\top \mA_i x - b_i^\top x$ for all $i \in [n],$ with $\mA_i \in \textnormal{Sym}^{d}_{+}$ and $b_i \in \R^d$. Let Assumption~\ref{ass:inter} hold and consider two adaptive extrapolation strategies: 
    \begin{enumerate}
        \item (\algname{FedExProx-GraDS}) Set
            \begin{align*}
                \squeeze \alpha_k = \alpha_{k}^{\algname{GraDS}}(x_k) \eqdef \frac{\frac{1}{n}\sum_{i=1}^n \norm{\nabla M^{\gamma}_{f_i}(x_{k})}^{2}}{\norm{\frac{1}{n}\sum_{i=1}^n \nabla M^{\gamma}_{f_i}(x_{k})}^{2}}
                \geq 1.
            \end{align*}
            Then, the iterates of Algorithm \ref{algorithm:fedexprox} satisfy
            % \begin{eqnarray}\label{eq:grads_iter}
            %     \squeeze \norm{x_K-\Pi(x_K)}^{2}
            %     &\leq& \squeeze \parens{1 - \min\limits_{k=0,\ldots,K-1} \alpha_{k} \gamma \frac{2 + \gamma L_{\max}}{1 + \gamma L_{\max}} \mu^{+}_{\gamma}}^K \norm{x_0 - \Pi(x_0)}^2.
            % \end{eqnarray}
            \small{\begin{eqnarray}\label{eq:grads_iter}
                \squeeze \norm{x_K-\Pi(x_K)}^{2} \leq \parens{1 - C_{\algname{G}}}^K \norm{x_0 - \Pi(x_0)}^2,
            \end{eqnarray}}
            where $C_{\algname{G}} \eqdef \min\limits_{k=0,\ldots,K-1} \alpha_{k} \gamma \frac{2 + \gamma L_{\max}}{1 + \gamma L_{\max}} \mu^{+}_{\gamma}$.
        \item (\algname{FedExProx-StoPS}) Set
            \small{\begin{align*}
                \squeeze \alpha_k = \alpha_{k}^{\algname{StoPS}}(x_k) \eqdef \frac{\frac{1}{n}\sum_{i=1}^{n} \parens{M^{\gamma}_{f_i}(x_{k})-\inf  M^{\gamma}_{f_i}}}{\gamma\norm{\frac{1}{n}\sum_{i=1}^{n} \nabla  M^{\gamma}_{f_i}(x_{k})}^{2}}
                \geq \frac{1}{2\gamma L_{\gamma}}.
            \end{align*}}
            Then, the iterates of Algorithm \ref{algorithm:fedexprox} satisfy
            % \begin{eqnarray}\label{eq:stops_iter}
            %     \squeeze \norm{x_K-\Pi(x_K)}^{2}
            %     &\leq& \squeeze \parens{1 - \frac{3}{2} \min\limits_{k=0,\ldots,K-1} \alpha_{k} \gamma \mu^{+}_{\gamma}}^K \norm{x_0 - \Pi(x_0)}^2.
            % \end{eqnarray}
            \small{\begin{eqnarray}\label{eq:stops_iter}
                \squeeze \norm{x_K-\Pi(x_K)}^{2}
                &\leq& \squeeze \parens{1 - C_{\algname{S}}}^K \norm{x_0 - \Pi(x_0)}^2,
            \end{eqnarray}}
            where $C_{\algname{S}} \eqdef \frac{3}{2} \min\limits_{k=0,\ldots,K-1} \alpha_{k} \gamma \mu^{+}_{\gamma}$.
    \end{enumerate}
\end{restatable}

\begin{remark}
    Similar to the observations by \citet{li2024power}, we see that, unlike in the constant extrapolation case (see Theorem \ref{thm:quad}), \algname{FedExProx} with adaptive extrapolation  benefits from \textit{semi-adaptivity} to the smoothness constant. Specifically, it converges for any~$\gamma > 0$, and since $\gamma \mu^{+}_{\gamma}$ is bounded above (see \eqref{eq:lmin+_eig}), it suffices to choose a large enough~$\gamma$ to achieve the optimal performance.
\end{remark}

\begin{remark}
    Bounding $\min_{k=0,\ldots,K-1} \alpha_k$ by $\frac{1}{2\gamma L_{\gamma}}$, inequality \eqref{eq:stops_iter} implies that
    \begin{eqnarray*}
        \squeeze \norm{x_K-\Pi(x_K)}^{2}
        &\leq& \squeeze \parens{1 - \frac{3}{4} \frac{\mu^{+}_{\gamma}}{L_{\gamma}}}^K \norm{x_0 - \Pi(x_0)}^2,
    \end{eqnarray*}
    and hence Theorem \ref{thm:quad_iter_ada} guarantees convergence of \algname{FedExProx-StoPS} in $\cO\left(\nicefrac{L_{\gamma}}{\mu^{+}_{\gamma}} \times \log \nicefrac{1}{\varepsilon}\right)$
    % \begin{align*}
    %     \squeeze \cO\left(\frac{L_{\gamma}}{\mu^{+}_{\gamma}} \log \frac{1}{\varepsilon}\right)
    % \end{align*}
    iterations, regardless of the choice of the stepsize $\gamma$. This matches the guarantee from Theorem~\ref{thm:quad}, but without requiring prior knowledge of the optimal extrapolation parameter $\alpha = \nicefrac{1}{\gamma L_{\gamma}}$. The benefit comes with the trade-off of needing to know the minimum of the average of Moreau envelopes.
\end{remark}

\section{Better Theory with P{\L} Condition}\label{sec:pl}

The result from Section~\ref{sec:tight} can be extended to solving problem \eqref{eq:main_problem} for general functions that satisfy the Polyak-\L ojasiewicz (P\L) condition.
\begin{assumption}
\label{ass:pl_condition_moreau}
The function $M^{\gamma}$ satisfies P\L-condition, i.e., there exists $\mu^{+}_{\gamma}$ such that $$\squeeze \frac{1}{2} \norm{\nabla M^{\gamma}(x)}^2 \geq \mu^+_{\gamma} \left(M^{\gamma}(x) - M^{\gamma}(x_*)\right) \qquad \forall x \in \R^d.$$
\end{assumption}
This assumption holds with $\mu^{+}_{\gamma} \geq \nicefrac{\mu^{+}}{(4 (1 + \gamma L_{\max}))}$ if the function $f$ satisfies P\L\,condition with a constant $\mu^{+}$ (Lemma \ref{lemma:pl_vs_mu+}). However, the choice $\mu^{+}_{\gamma} = \nicefrac{\mu^{+}}{(4 (1 + \gamma L_{\max}))}$ is loose and leads to problems described in Theorem~\ref{thm:pessimistic}.
With this assumption in place, we can present the theorem.

\begin{theorem}\label{thm:fedexprox_pl_iter}
    Let Assumptions~\ref{ass:convex}, \ref{ass:local_lipschitz_constant}, \ref{ass:inter}, and \ref{ass:pl_condition_moreau} hold.
    For all $\gamma > 0,$ \algname{FedExProx} (Algorithm~\ref{algorithm:fedexprox}) with $\alpha = \nicefrac{1}{\gamma L_{\gamma}}$ finds $\bar{x}$ such that $\mathbb{E}[f(\bar{x})] - f(x_*) \leq \varepsilon$ in
    \begin{align*}
        \squeeze \cO\left(\frac{L_{\gamma}}{\mu^{+}_{\gamma}} \log \frac{1}{\varepsilon}\right)
    \end{align*}
    iterations, where $L_{\gamma}$ is a smoothness constant of $M^{\gamma}$ and $\mu^{+}_{\gamma}$ is the P{\L} constant.
\end{theorem}

\begin{remark}
    In practice, solving the local problems exactly is often infeasible. Instead, we can only calculate an inexact proximal operator that approximates the desired quantity within a certain accuracy budget. To accommodate this, we extend the above analysis to the scenario where the clients can only compute updates $\prox^{\delta}_{\gamma f_i} (x)$ such that $\|\prox^{\delta}_{\gamma f_i} (x) - \prox_{\gamma f} (x)\|^2 \leq \delta$ (see Section~\ref{sec:inexact}).
\end{remark}

\begin{remark}
    \citet{li2024power} also establish a linear rate for \algname{FedExProx}. However, their result relies on the assumption that the function $f$ is strongly convex, and hence that \eqref{eq:main_problem} has a unique solution. In contrast, our theorem is more general, as it allows for the possibility of multiple solutions.
\end{remark}

A similar result can be established in the partial participation scenario.
\begin{theorem}\label{thm:fedexprox_pl_iter_stoch}
    Let Assumptions~\ref{ass:convex}, \ref{ass:local_lipschitz_constant}, \ref{ass:inter}, and \ref{ass:pl_condition_moreau} hold.
    For all $\gamma > 0,$ \algname{FedExProx} with nice sampling (Algorithm~\ref{algorithm:batch_fedexprox}) and $\alpha = \nicefrac{1}{\gamma L_{\gamma}}$ finds $\bar{x}$ such that $\mathbb{E}[f(\bar{x})] - f(x_*) \leq \varepsilon$ in
    \begin{align*}
        \squeeze \cO\left(\frac{L_{\gamma, S}}{\mu^{+}_{\gamma}} \log \frac{1}{\varepsilon}\right)
    \end{align*}
    iterations, where $L_{\gamma, S} \eqdef \parens{\frac{n-S}{S(n-1)} \frac{L_{\max}}{1+\gamma L_{\max}} + \frac{n(S-1)}{S(n-1)} L_{\gamma}}$, $L_{\gamma}$ is a smoothness constant of $M^{\gamma}$ and~$\mu^{+}_{\gamma}$ is the P{\L} constant.
\end{theorem}

Note that the complexities in Theorems~\ref{thm:fedexprox_pl_iter} and~\ref{thm:fedexprox_pl_iter_stoch} are entirely analogous to those in Theorems~\ref{thm:quad} and~\ref{thm:quad_iter_stoch}, with the only distinction being the substitution of the smallest non-zero eigenvalue of $\nabla^2 M^{\gamma}$ with the P\L\,constant.

\begin{figure*}[t]
    \centering
    \includegraphics[width=0.95\textwidth]{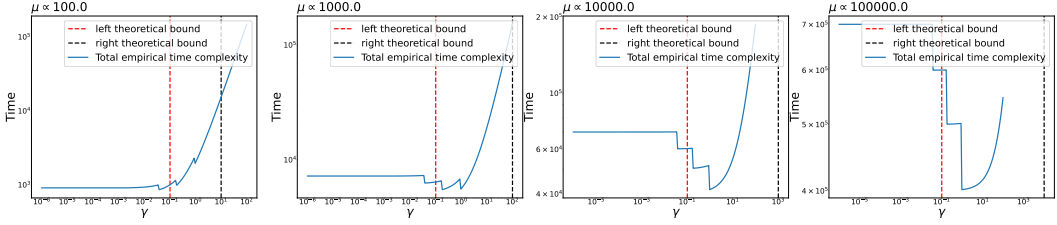}
    \caption{Empirical time complexities of \algname{FedExProx} on a quadratic optimization task.}
    \label{fig:full_quadratics}
\end{figure*}

\subsection{Why do we get a tighter analysis?}

The story behind our theoretical improvements is no less important than the improvements themselves, as it offers valuable insights and can be instructive for future research.
To understand why the new analysis yields stronger guarantees, we need to examine the reasoning behind Theorem~\ref{thm:pessimistic}.
When proving the result, we noticed that the main reason why \eqref{eq:pess_inequality} was true was the dependence of the complexity on the $L_{\max}$ factor. At the same time, the complexity of \algname{GD} depends only on $L,$ and there exist many examples when $L_{\max} \gg L.$ The question was: why does the original analysis of \algname{FedExProx} involve $L_{\max}$, and can one circumvent this dependency? The reason behind it is the reliance of the proofs by \citet{li2024power} on Lemma~\ref{lemma:m_to_f}, which establishes the inequality $M^{\gamma}(x) - M^{\gamma}(x_*) \geq \frac{1}{1+\gamma L_{\max}} \parens{f(x) - f(x_*)}$. This result appears to be essential for deriving the convergence result $\mathbb{E}[f(\bar{x})] - f(x_*) \leq \varepsilon$ (where $\bar{x} \in \R^d$ is the output of \algname{FedExProx}).
Our main idea is to instead obtain convergence in terms of \emph{distances}, i.e., to establish that $\mathbb{E}[\norm{\bar{x} - x_*}^2] \leq \varepsilon$, and then translate it to $\mathbb{E}[f(\bar{x})] - f(x_*) \leq \varepsilon$ using the $L$--smoothness of the function $f$. In this way, one can avoid using Lemma~\ref{lemma:m_to_f}, ultimately obtaining much more favourable convergence guarantees.

\section{Experimental Highlights}

This section presents highlights from illustrative experiments; additional details and results--including those on the \texttt{ARCENE} dataset \citep{arcene_167}--are provided in \Cref{sec:experiments}. We consider quadratic objectives of the form $f(x) = \frac{1}{n}\sum_{i=1}^{n}\frac{1}{2}x^{\top} \mA_{i}x,$ where $\mA_{i} \in \textnormal{Sym}^{7}_{+}$, $i\in[n]$ are random positive semidefinite matrices with minimum eigenvalue equal to zero. 
Each worker $i\in[14]$ computes proximal mappings using \algname{GD} with stepsize $\nicefrac{1}{\bar{L}_i}$.
The extrapolation parameter $\alpha$ is set to its optimal value from Theorem~\ref{thm:quad}. To validate our theory, we study the empirical time complexity of \algname{FedExProx} as a function of~$\gamma$.
In accordance with the setup described in Section \ref{sec:time_cond}, we assume that one local iteration of \algname{GD} takes $\tau$ seconds (without loss of generality, $\tau=1$). Thus, the time needed by worker $i$ to find $\prox_{\gamma f_{i}}(x_k)$ at global iteration $k$ is proportional to $\tau \times n_{ik},$ where $n_{ik}$ is the number of \algname{GD} iterations needed to find $\prox_{\gamma f_{i}}(x_{k})$ to a given accuracy. In the full participation case, the total empirical time complexity is $\sum_{k=0}^{K-1} (\ct + \tau \max_{i \in [n]} n_{ki})$, 
% \begin{align}
%     \squeeze\sum_{k=0}^{K-1} \left(\ct + \tau \max_{i \in [n]} n_{ki}\right)
% \end{align}
where $K$ is the number of global iterations needed for \algname{FedExProx} to converge to the desired accuracy and $\ct$ is the communication time.
The results are presented in Figure~\ref{fig:full_quadratics}. The dashed lines represent the theoretical bounds from Theorem~\ref{thm:fedexprox_quad_time}, within which the optimal~$\gamma$ is expected to lie. One can see that when $\ct$ is relatively small ($\ct \propto 100$), the best choice of~$\gamma$ is near $0.$ However, as the communication cost $\ct$ increases, the best~$\gamma$ shifts to values greater than $0.1$.
A distinctive U-shape emerges, indicating the nontrivial optimal choice of~$\gamma.$ These observations are fully consistent with our theoretical predictions.

\section{Conclusion}

In this work, we revisit the extrapolated parallel proximal method \algname{FedExProx}, an algorithm that has shown strong empirical performance but whose theoretical guarantees have so far lagged behind practice. Upon re-examining the state-of-the-art analysis, we find that its guarantees are overly pessimistic and that significantly stronger results are achievable.
To address this gap, we develop a novel analytical framework for non-strongly convex quadratic and P{\L} cases, yielding substantially improved convergence guarantees and a clearer picture of the method’s true performance. To put these findings in context, we compare \algname{FedExProx} with vanilla \algname{GD}. This comparison is primarily pedagogical but reveals a striking insight: previous analyses could not even show an advantage over this simplest baseline. In contrast, our work is the first to rigorously demonstrate that \algname{FedExProx} can outperform \algname{GD}, bridging a key gap between empirical observations and theoretical understanding.
Although our focus is on relatively simple quadratic problems, experiments suggest that the phenomena we describe extend well beyond this setting. These results indicate that the benefits of extrapolation in FL hold in broader optimization scenarios, motivating future research to relax our assumptions and establish similar guarantees in more general and practical settings. We view this work as a foundational stepping stone toward a deeper understanding of extrapolation in FL, paving the way for advancing both the theory and practice of federated optimization.

\newpage

\subsubsection*{Acknowledgments}

The research reported in this publication was supported by funding from King Abdullah University of Science and Technology (KAUST): i) KAUST Baseline Research Scheme, ii) CRG Grant ORFS-CRG12-2024-6460, and iii) Center of Excellence for Generative AI, under award number 5940.

\bibliography{iclr2026_conference}
\bibliographystyle{iclr2026_conference}

\newpage

\appendix
\section*{Appendix}

\tableofcontents

\newpage

\section{Background}

Before discussing the contributions of this work, we provide formal definitions and essential facts that will be repeatedly referenced in the proofs. A summary of the commonly used notation is provided in Table \ref{table:notation}.

\begin{definition}[Proximal operator]
    The \emph{proximal operator} $\prox_{\gamma f} : \R^d \rightarrow \R^d$ of $f$ is defined by $$\prox_{\gamma f} (x) = \underset{z\in\R^d}{\arg\min} \brac{f(z) + \frac{1}{2\gamma} \norm{z - x}^2},$$
    where $\norm{\cdot}$ is the standard Euclidean norm.
\end{definition}

\begin{definition}[Moreau envelope]
    The \emph{Moreau envelope} of an extended-real-valued function $f: \R^d \rightarrow \R^d \cup \{\infty\}$ with stepsize $\gamma > 0$ is the function $M^{\gamma}_{f} \,:\, \R^d \rightarrow \R^d$ such that 
    \begin{align*}
        M^{\gamma}_{f}(x) = \min_{y\in\R^d} \brac{f(y) + \frac{1}{2\gamma} \norm{y-x}^2}
    \end{align*}
    for all $x \in \R^d.$
\end{definition}

\subsection{Useful Facts about Moreau Envelope}

We will often rely on several useful properties of the Moreau envelopes, summarized below. In what follows, we denote
$$M^{\gamma}(x) \eqdef \frac{1}{n}\sum_{i=1}^n M_{f_i}^{\gamma}(x).$$

\begin{lemma}[\citet{beck2017first} Theorem 6.60]\label{lemma:grad_moreau}
    Let $f:\R^{d}\to \R\cup\{+\infty\}$ be proper, closed and convex. Then its Moreau envelope $M_{f}^{\gamma}$ satisfies
    \begin{align*}
        \nabla M_{f}^{\gamma}(x) = \frac{1}{\gamma}(x-\prox _{\gamma f}(x))
    \end{align*}
    for all $x\in\R^{d}$, for any $\gamma>0$.
\end{lemma}

\begin{lemma}[\citep{li2024power} Lemma $4$]\label{lemma:moreau_smooth}
    Let $f:\R^{d}\to\R$ be convex and $L$--smooth. Then $M_{f}^{\gamma}$ is $\frac{L}{1+\gamma L}$--smooth.
\end{lemma}

\begin{lemma}[\citep{li2024power} Lemma $5$]\label{lemma:mini_equiv_local}
    Let $f:\R^d\mapsto\R\cup\{+\infty\}$ be a proper, closed and convex function. Then, for all $\gamma>0$, $f$ and $M_f^{\gamma}(x)$ have the same set of minimizers and minimum.
\end{lemma}

\begin{lemma}[\citep{beck2017first} Theorem $6.55$]\label{lemma:moreau_convex}
    Let $f:\R^d\mapsto\R\cup\{+\infty\}$ be a proper, closed and convex function. Then $M_{f}^{\gamma}$ is also a convex function.
\end{lemma}

\begin{lemma}[\citep{li2024power} Lemma $7$]\label{lemma:m_gamma_convex_smooth}
    Let the functions $f_i:\R^d\mapsto\R\cup\{+\infty\}$, $i\in[n]$, be proper, closed, convex and $L_i$--smooth. Then $M^{\gamma}$ is convex and $L_{\gamma}$--smooth with
    \begin{align*}
        \frac{1}{n^2} \sum_{i=1}^n \frac{L_i}{1+\gamma L_i}\leq L_{\gamma} \leq \frac{1}{n} \sum_{i=1}^n \frac{L_i}{1+\gamma L_i}.
    \end{align*}
\end{lemma}

\begin{lemma}[\citep{li2024power} Lemma $8$]\label{lemma:mini_equiv_global}
    Let the functions $f_i:\R^d\mapsto\R\cup\{+\infty\}$, $i\in[n]$, be proper, closed and convex, and suppose that Assumption \ref{ass:inter} holds. Then, for all $\gamma\in(0,\infty)$, the function $f(x)=\frac{1}{n}\sum_{i=1}^n f_i(x)$ has the same set of minimizers and minimum as the function $M^{\gamma}(x)$.
\end{lemma}

\begin{lemma}\label{lemma:m_to_f}[\citet{li2024power} Lemma 10]
    Let Assumptions \ref{ass:convex}, \ref{ass:local_lipschitz_constant} and \ref{ass:inter} hold. Then, for any $x_*\in\cX_*$ and for all $x\in\R^d$, it holds that
    \begin{align*}
        M^{\gamma}_{f_i}(x) - M^{\gamma}_{f_i}(x_*) \geq \frac{1}{1+\gamma L_i} \parens{f_i(x) - f_i(x_*)}.
    \end{align*}
    Consequently,
    \begin{align*}
        M^{\gamma}(x) - M^{\gamma}(x_*) \geq \frac{1}{1+\gamma L_{\max}} \parens{f(x) - f(x_*)}.
    \end{align*}
\end{lemma}

The next lemma will play a role in the analysis of the method when assuming the Polyak-{\L}ojasiewicz~(P{\L}) condition.
\begin{assumption}[Polayk-{\L}ojasiewicz condition]\label{ass:pl}
    The function $f$ satisfies P\L\,condition, i.e., there exists $\ct>0$ such that $$\frac{1}{2} \norm{\nabla f(x)}^2 \geq \ct \left(f(x) - f(x_*)\right)$$ for all $x \in \R^d.$
\end{assumption}

\begin{lemma}\label{lemma:pl_vs_mu+}
    Let the function $f$ satisfy Assumption~\ref{ass:pl} with parameter $\mu^{+}$, and suppose that Assumptions~\ref{ass:convex}, \ref{ass:local_lipschitz_constant} and \ref{ass:inter} hold. Then, Assumption \ref{ass:pl_condition_moreau} holds with parameter $$\mu^{+}_{\gamma} \geq \frac{\mu^{+}}{1 + \gamma L_{\max}}.$$
\end{lemma}
\begin{proof}
    If $f$ satisfies Assumption~\ref{ass:pl}, then under Assumption~\ref{ass:convex}, Polayk-{\L}ojasiewicz condition implies
    \begin{align}
        \label{eq:YVasxCFrLhuOK}
        f(x) - f(x_*) \geq \frac{\mu^{+}}{2} \norm{x - \Pi(x)}^2 \quad \forall x \in \R^d,
    \end{align}
    \citep{karimi2016linear}, where $\Pi(x)$ is the projection of $x$ onto the solution set $\cX_*$ of $f$ and $x_*\in\cX_*$. Then, using Lemma~\ref{lemma:m_to_f}, we get
    \begin{align*}
        M^{\gamma}(x) - M^{\gamma}(x_*) \geq \frac{1}{1+\gamma L_{\max}} \parens{f(x) - f(x_*)} \overset{\eqref{eq:YVasxCFrLhuOK}}{\geq} \frac{\mu^{+}}{2 (1+\gamma L_{\max})} \norm{x - \Pi(x)}^2.
    \end{align*}
    Due to Lemma~\ref{lemma:mini_equiv_global}, $M^{\gamma}$ and $f$ share the set of minimizers. Thus, $\Pi(x)$ is the projection of $x$ onto the solution set $\cX_*$ of $M^{\gamma}.$ 
    Using convexity, we get
    \begin{align*}
        M^{\gamma}(\Pi(x)) \geq M^{\gamma}(x) + \inp{\nabla M^{\gamma}(x)}{\Pi(x) - x}
    \end{align*}
    and
    \begin{align*}
        M^{\gamma}(x) - M^{\gamma}(\Pi(x)) 
        &\leq \inp{\nabla M^{\gamma}(x)}{x - \Pi(x)} \leq \norm{\nabla M^{\gamma}(x)} \norm{x - \Pi(x)} \\
        &\leq \norm{\nabla M^{\gamma}(x)} \sqrt{\frac{2 (1+\gamma L_{\max})}{\mu^{+}}}\sqrt{M^{\gamma}(x) - M^{\gamma}(\Pi(x))}.
    \end{align*}
    Therefore
    \begin{align*}
        \frac{\mu^{+}}{4 (1+\gamma L_{\max})}\left(M^{\gamma}(x) - M^{\gamma}(\Pi(x))\right) \leq \frac{1}{2}\norm{\nabla M^{\gamma}(x)}^2,
    \end{align*}
    meaning that $M^{\gamma}$ satisfies P\L\,condition with parameter $\frac{\mu^{+}}{4 (1+\gamma L_{\max})}$.
\end{proof}

\subsection{Nice Sampling}\label{sec:nice_sampling}

Let us formally introduce the sampling strategy used.
Fix a minibatch size $S\in[n]$ and let $\cS$ be a random subset of $[n]$ of size $S$, chosen uniformly at random from all ${n \choose S}$ subsets of $[n]$ of this size. Such a random set $\cS$ is known in the literature under the name \textit{$S$-nice sampling} \citep{richtarik2016parallel}.

In the proofs, we will rely on the following useful lemma:

\begin{lemma}\label{lemma:nice_m}
    Fix $\mB_1,\ldots,\mB_n \in \R^{k\times l}$ and let $\cS$ be an $S$-nice sampling of the indices $[n]$. Then
    \begin{eqnarray*}
        &&\hspace{-1cm}\Exp{\parens{\frac{1}{S} \sum_{i\in\cS} \mB_i}^\top \parens{\frac{1}{S} \sum_{i\in\cS} \mB_i}} \\
        &=& \frac{n-S}{S(n-1)} \frac{1}{n} \sum_{i=1}^n \mB_i^\top \mB_i + \frac{n(S-1)}{S(n-1)} \parens{\frac{1}{n} \sum_{i=1}^n \mB_i}^\top \parens{\frac{1}{n} \sum_{i=1}^n \mB_i}.
    \end{eqnarray*}
\end{lemma}
\begin{proof}
    Define
    \begin{align*}
        \chi_{i} \eqdef
        \begin{cases}
            1 & i\in\cS, \\
            0 & i\notin\cS. \\
        \end{cases}
    \end{align*}
    Then
    \begin{eqnarray*}
        &&\hspace{-1cm}\Exp{\parens{\frac{1}{S} \sum_{i\in\cS} \mB_i}^\top \parens{\frac{1}{S} \sum_{i\in\cS} \mB_i}} \\
        &=& \Exp{\parens{\frac{1}{S} \sum_{i=1}^n \chi_{i} \mB_i}^\top \parens{\frac{1}{S} \sum_{i=1}^n \chi_{i} \mB_i}} \\
        &=& \frac{1}{S^2} \Exp{\sum_{i=1}^n \chi_{i}^2 \mB_i^\top \mB_i + \sum_{i\neq j} \chi_{i} \chi_{j} \mB_i^\top \mB_j} \\
        &=& \frac{1}{S^2} \parens{\sum_{i=1}^n \Exp{\chi_{i}} \mB_i^\top \mB_i
        + \sum_{i\neq j} \Exp{\chi_{i} \chi_{j}} \mB_i^\top \mB_j} \\
        &=& \frac{1}{S^2} \parens{\frac{S}{n} \sum_{i=1}^n \mB_i^\top \mB_i
        + \frac{S(S-1)}{n(n-1)} \sum_{i\neq j} \mB_i^\top \mB_j} \\
        &=& \frac{1}{S} \parens{\frac{1}{n} \sum_{i=1}^n \mB_i^\top \mB_i
        + \frac{S-1}{n(n-1)} \parens{\parens{\sum_{i=1}^n \mB_i}^\top \parens{\sum_{i=1}^n \mB_i} - \sum_{i=1}^n \mB_i^\top \mB_i}} \\
        &=& \frac{n-S}{S(n-1)} \frac{1}{n} \sum_{i=1}^n \mB_i^\top \mB_i
        + \frac{n(S-1)}{S(n-1)} \parens{\frac{1}{n} \sum_{i=1}^n \mB_i}^\top \parens{\frac{1}{n} \sum_{i=1}^n \mB_i}.
    \end{eqnarray*}
\end{proof}

\subsection{Note on Other Sampling Strategies}

As discussed in Section \ref{sec:pp}, Algorithm \ref{algorithm:batch_fedexprox} can handle virtually any (unbiased) sampling technique, with nice sampling used here as an illustrative example. For a comprehensive overview of other potential sampling strategies, we direct the reader to \citet{tyurin2022sharper}, which offers a general framework for analyzing a broad range of sampling mechanisms.

\subsection{Useful Lemmas}

\begin{lemma}[\citep{karimi2016linear} Theorem $2$]\label{lemma:star_conv}
    Assume that the function $f$ is convex, $L$--smooth, and satisfies the Polyak-{\L}ojasiewicz condition with parameter $\ct$. Then
    \begin{align*}
        f(x) - f(x_*) \geq \frac{\ct}{2} \norm{x - \Pi(x)}^2,
    \end{align*}
    for all $x\in\R^d$, where $\Pi(x)$ is the projection of $x$ onto the solution set $\cX_*$ and $x_*\in\cX_*$.
\end{lemma}

\begin{lemma}[\citep{nesterov2018lectures}]\label{lemma:grad_breg}
    Let $D_f(x, y) \eqdef f(x) - f(y) - \inp{\nabla f(y)}{x-y}$ be the Bregman divergence of the function $f$.
    If $f$ is convex, $L$--smooth and differentiable, then
    \begin{align*}
        \frac{1}{L} \norm{\nabla f(x) - \nabla f(y)}^2 \leq D_f(x, y) + D_f(y, x)
    \end{align*}
    for any $x, y \in\R^d$.
\end{lemma}

\newpage

\section{FedExProx as (S)GD on Moreau Envelope Reformulation}\label{sec:fedexprox_sgd}

\begin{algorithm}[t]
    \caption{\algname{FedExProx}}
    \label{algorithm:fedexprox}
    \begin{algorithmic}[1]
    \STATE \textbf{Parameters:} stepsize $\gamma > 0$, extrapolation parameter $\alpha_t>0$, starting point $x_0 \in \R^d$
    \FOR{$k = 0, 1, 2, \dots$}
        \STATE $x_{k+1} = x_k + \alpha_k \parens{\frac{1}{n} \sum_{i=1}^{n} \prox_{\gamma f_i} (x_k) - x_k}$
    \ENDFOR
    \end{algorithmic}
\end{algorithm}

The key idea behind the original analysis of \algname{FedExProx} is to rewrite the step
\begin{align*}
    x_{k+1} = x_{k} + \alpha_k\left(\frac{1}{n}\sum_{i=1}^{n}\prox_{\gamma f_i}(x_{k}) - x_{k}\right)
\end{align*}
in terms of the Moreau envelopes of the local functions.
Indeed, if the functions $f_i$ are proper, closed and convex for all $i\in[n]$, by Lemma \ref{lemma:grad_moreau}, the update rule of Algorithm \ref{algorithm:fedexprox} can be rewritten as
\begin{align*}
    x_{k+1} = x_{k} - \alpha_k \gamma \frac{1}{n}\sum_{i=1}^{n} \nabla M_{f_i}^{\gamma}(x_{k}).
\end{align*}
Analogously, in the partial participation case (Algorithm \ref{algorithm:batch_fedexprox}), we have
\begin{align*}
    x_{k+1} = x_{k} - \alpha_k \gamma \frac{1}{S}\sum_{i\in\cS_k} \nabla M_{f_i}^{\gamma}(x_{k}).
\end{align*}
It follows that Algorithms \ref{algorithm:fedexprox} and \ref{algorithm:batch_fedexprox} are equivalent to \algname{GD} and minibatch \algname{SGD} for minimizing $$M^{\gamma}(x) \eqdef \frac{1}{n}\sum_{i=1}^{n} M_{f_i}^{\gamma}(x)$$ with stepsize $\alpha_k\gamma$.
    
\newpage

\section{FedExProx for Quadratics}\label{sec:fedexprox_quad}

We begin by formally introducing the problem setup and the notation used throughout the proofs.

We are interested in solving the quadratic optimization problem
\begin{align}\label{eq:problem_quad}
    \min_{x\in\R^n} \brac{f(x) \eqdef \frac{1}{n} \sum_{i=1}^n \parens{\frac{1}{2} x^\top \mA_i x - b_i^\top x}},
\end{align}
where $\mA_i \in \textnormal{Sym}^{d}_{+} \eqdef \{\mX \in \R^{d\times d} \,|\, \mX=\mX^\top, \mX\succeq 0\}$ and $b_i \in \R^d$, assuming the interpolation regime (Assumption \ref{ass:inter}), which in this case states that there exists $x_*$ such that $\mA_{i}x_* = b_i$ for all~$i\in[n]$.

We denote the set of minimizers by $\cX_* \eqdef \{x\in\R^d: \frac{1}{n} \sum_{i=1}^{n} (\mA_{i}x - b_i) = 0\}$ and let $\Pi(\cdot)$ be the projection onto $\cX_*$:
\begin{align}\label{eq:projcetion}
    \Pi(x) = x - \parens{\frac{1}{n} \sum_{i=1}^{n} \mA_i}^{\dagger} \parens{\frac{1}{n} \sum_{i=1}^{n} (\mA_i x - b_i)},
\end{align}
where $\dagger$ denotes the Moore-Penrose pseudoinverse.

Using eigendecomposition, each matrix $\mA_i$, $i\in[n]$ can be decomposed as $\mA_i = \mQ_i \blambda_i \mQ_i^\top$, where $\mQ_i =[q_{i,1},\ldots,q_{i,d}] \in \R^{d\times d}$ is the orthogonal matrix whose columns are eigenvectors of $\mA_i$ and $\blambda_i \in \R^{d\times d}$ is a diagonal matrix whose diagonal elements are the corresponding eigenvalues of~$\mA_i$, i.e., $\blambda_i = \textnormal{diag}\parens{\lambda_{1}(\mA_i), \ldots, \lambda_{d}(\mA_i)}$, where $0 = \lambda_{1}(\mA_i) \leq \lambda_{2}(\mA_i) \leq \ldots \leq \lambda_{d}(\mA_i)$. Let~$\lambda_{\min}^+(\mA_i)$ denote the smallest non-zero eigenvalue of $\mA_i$ and~$\lambda_{\max}(\mA_i)$--the largest eigenvalue of $\mA_i$.

In general, the largest eigenvalue, smallest eigenvalue, and smallest positive eigenvalue of any matrix~$\mB$ will be denoted by $\lambda_{\max}(\mB)$, $\lambda_{\min}(\mB)$ and $\lambda_{\min}^+(\mB)$, respectively.

The proximal operators in this setting are
\begin{align}\label{eq:quad_prox}
    \prox_{\gamma f_i}(x)= \parens{\gamma \mA_i + \mI}^{-1} \parens{x + \gamma b_i}.
\end{align}
Thus, by Lemma \ref{lemma:grad_moreau},
\begin{align}\label{eq:nabla_mfi}
    \nabla M^{\gamma}_{f_{i}}(x) = \frac{1}{\gamma} \parens{x - \parens{\gamma \mA_i + \mI}^{-1} \parens{x + \gamma b_i}}
    = \frac{1}{\gamma} \parens{\mI - \parens{\gamma \mA_i + \mI}^{-1} } (x - x_*)
\end{align}
for any $x_*\in\cX_*$, where we use the fact that $\mA_{i}x_* = b_{i}$.
Furthermore, it can easily be shown that
\begin{align}\label{eq:fedexprox_mfi}
    M^{\gamma}_{f_{i}}(x) & = \frac{1}{2\gamma}\left(x^{T}(\mI-(\gamma \mA_{i} + \mI)^{-1})x - 2\gamma b_{i}^{T}(\gamma \mA_{i} + \mI)^{-1}x - \gamma^{2}b_{i}^{T}(\gamma \mA_{i} + \mI)^{-1}b_{i}\right) \nonumber \\
    & =\frac{1}{2\gamma} \parens{(x-x_*)^{T}(\mI-(\gamma \mA_{i} + \mI)^{-1})(x-x_*) - \gamma x_*^{T} \mA_{i} x_*},
\end{align}
and hence
\begin{eqnarray}\label{eq:fedexprox_m}
    M^{\gamma}(x) & = & \frac{1}{2} x^{T} \left(\frac{1}{n} \sum_{i=1}^{n} \frac{1}{\gamma} (\mI - (\gamma \mA_{i} + \mI)^{-1})\right) x - \left(\frac{1}{n}\sum_{i=1}^{n}(\gamma \mA_{i} + \mI)^{-1}b_{i}\right)^{T}x \nonumber \\
    &&- \frac{\gamma}{2n}\sum_{i=1}^{n}b_{i}^{T}(\gamma \mA_{i} + \mI)^{-1}b_{i} \nonumber\\
    &=& \frac{1}{2} x^{T} \mM x - \left(\frac{1}{n}\sum_{i=1}^{n}(\gamma \mA_{i} + \mI)^{-1}b_{i}\right)^{T}x- \frac{\gamma}{2n}\sum_{i=1}^{n}b_{i}^{T}(\gamma \mA_{i} + \mI)^{-1}b_{i}
\end{eqnarray}
where $\mM \eqdef \frac{1}{n}\sum_{i=1}^{n}\frac{1}{\gamma}(\mI - (\gamma \mA_{i} + \mI)^{-1})$ is the Hessian matrix of the function $M^{\gamma}$.
This can equivalently be written as
\begin{eqnarray}\label{eq:fedexprox_m2}
    M^{\gamma}(x) &=& \frac{1}{2} (x-x_*)^{T} \mM (x-x_*) - x_*^{T}\left(\frac{1}{2n}\sum_{i=1}^{n} \mA_{i}\right)x_*
\end{eqnarray}
for any $x_*\in\cX_*$.
We shall denote by $\lambda_{\min}^+(\gamma)$ and $\lambda_{\max}(\gamma)$ the smallest non-zero eigenvalue and the largest eigenvalue of $\mM$, respectively.

\subsection{Properties of the Hessian}

Next, we introduce several important properties of the matrix $\mM$, which will be repeatedly used in the proofs that follow.

In the main part of the paper, we stick to the notation $L_{\gamma}$ and $\mu^{+}_{\gamma}$ for easier comparison with the results by \citet{li2024power}. However, in the proofs concerning the quadratic case, we adopt a more intuitive eigenvalue-based notation.
\begin{fact}\label{fact:hess_m}
    Since $\mM = \nabla^2 M^{\gamma}$, it holds that $L_{\gamma} = \lambda_{\max}(\gamma)$ and $\mu^{+}_{\gamma} = \lambda_{\min}^+(\gamma)$.
\end{fact}
\begin{proof}
    The result is an immediate consequence of \eqref{eq:fedexprox_m}.
\end{proof}

\begin{fact}\label{fact:m_eigen_evals}
    Let $\lambda_{\max}(\gamma)$ and $\lambda_{\min}^{+}(\gamma)$ be the smallest non-zero eigenvalue and the largest eigenvalue of the matrix
    \begin{align*}
        \mM = \frac{1}{n}\sum_{i=1}^{n} \frac{1}{\gamma} \parens{\mI-(\gamma \mA_{i}+\mI)^{-1}},
    \end{align*}
    respectively. Then
    \begin{align}\label{eq:lmax_eig}
        \lambda_{\max}(\gamma) = \max_{\norm{x} \leq 1} \frac{1}{n}\sum_{i=1}^{n} x^\top \mQ_i \left[\frac{\lambda_{j}(\mA_{i})}{1 + \gamma \lambda_{j}(\mA_{i})}\right]_{jj} \mQ_i^\top x
    \end{align}
    and
    \begin{align}\label{eq:lmin+_eig}
        \lambda_{\min}^{+}(\gamma) = \min_{\norm{x} = 1, x \in \ker(\mA)^{\perp}} \frac{1}{n}\sum_{i=1}^{n} x^\top \mQ_i \left[\frac{\lambda_{j}(\mA_{i})}{1 + \gamma \lambda_{j}(\mA_{i})}\right]_{jj} \mQ_i^\top x,
    \end{align}
    where $\mA \eqdef \frac{1}{n} \sum_{i=1}^{n} \mA_i$ and $[b_j]_{jj}$ is the diagonal matrix with $b_j$ as the $j$th entry.
\end{fact}
\begin{proof}
    Using the eigendecomposition, we can write $\mM$ as
    \begin{align*}
        \mM = \frac{1}{n}\sum_{i=1}^{n} \frac{1}{\gamma} \mQ_i \parens{\mI-(\gamma \blambda_{i}+\mI)^{-1}} \mQ_i^\top
        = \frac{1}{n}\sum_{i=1}^{n} \mQ_i \left[\frac{\lambda_{j}(\mA_{i})}{1 + \gamma \lambda_{j}(\mA_{i})}\right]_{jj} \mQ_i^\top.
    \end{align*}
    The result follows from the identities
    \begin{align*}
        \lambda_{\max}(\gamma) = \max_{\norm{x} \leq 1}  x^\top \mM x
    \end{align*}
    and
    \begin{align*}
        \lambda_{\min}^{+}(\gamma) = \min_{\norm{x} = 1, x \in (\ker \mA)^{\perp}}  x^\top \mM x
    \end{align*}
    (see Lemma \ref{lemma:lmin+eigen}).
\end{proof}

\begin{fact}\label{fact:lg_l}
    % Let $L$ be the smoothness constant of $f$ and $L_{\gamma}$ be the smoothness constant of $M^{\gamma}$. Then $L_{\gamma} \leq L$.
    For any $\gamma \geq 0$, it holds that $\lambda_{\max}(\gamma) \leq \lambda_{\max}(\mA)$ and $\lambda_{\min}^+(\gamma) \leq \lambda_{\min}^+(\mA)$.
\end{fact}

\begin{proof}
    From Fact \ref{fact:m_eigen_evals}, we have
    \begin{align*}
        \lambda_{\max}(\gamma) \overset{\eqref{eq:lmax_eig}}{=} \max_{\norm{x} \leq 1} \frac{1}{n}\sum_{i=1}^{n} x^T \mQ_i \left[\frac{\lambda_{j}(\mA_{i})}{1 + \gamma \lambda_{j}(\mA_{i})}\right]_{jj} \mQ_i^T x,
    \end{align*}
    and similarly
    \begin{align*}
        \lambda_{\max}(\mA) = \max_{\norm{x} \leq 1} \frac{1}{n}\sum_{i=1}^{n} x^T \mQ_i \left[\lambda_{j}(\mA_{i})\right]_{jj} \mQ_i^T x.
    \end{align*}
    Since it always holds that $\frac{\lambda_{j}(\mA_{i})}{1 + \gamma \lambda_{j}(\mA_{i})} \leq \lambda_{j}(\mA_{i})$, we get 
    \begin{align*}
        \lambda_{\max}(\gamma) \leq \lambda_{\max}(\mA).
    \end{align*}
    The proof of the second inequality is analogous.
\end{proof}

\subsection{Proofs for Section \ref{sec:not_better}}

\THMPESSIMISTIC*
\begin{proof}
    We first prove that $L_{\gamma} (1+\gamma L_{\max}) \geq L = \lambda_{\max}\left(\frac{1}{n} \sum_{i=1}^n \mA_i\right)$ for all $\gamma \geq 0.$
    Since
    \begin{align*}
        \nabla^2 M_{f_{i}}(x) = \frac{1}{\gamma} \left(\mI - \parens{\gamma \mA_i + \mI}^{-1}\right),
    \end{align*}
    we have
    \begin{align*}
        L_{\gamma} = \lambda_{\max}\left(\frac{1}{n} \sum_{i=1}^n \frac{1}{\gamma} \left(\mI - \parens{\gamma \mA_i + \mI}^{-1}\right)\right).
    \end{align*}
    Therefore, using the fact that $L_{\max} = \max_{i} \lambda_{\max}(\mA_i)$, we obtain
    \begin{align}\label{eq:asdasfbfg}
        (1+\gamma L_{\max}) L_{\gamma}
        & = \parens{1+\gamma \max_{i} \lambda_{\max}(\mA_i)} \frac{1}{\gamma} \lambda_{\max}\left(\frac{1}{n} \sum_{i=1}^n \left(\mI - \left(\gamma \mA_i + \mI\right)^{-1}\right)\right) \\
        & =\lambda_{\max}\left(\frac{1}{n} \sum_{i=1}^n \frac{1 + \gamma\max_{i} \lambda_{\max}(\mA_i)}{\gamma} \left(\mI - \left(\gamma\mA_i + \mI\right)^{-1}\right)\right), \nonumber
    \end{align}
    and hence, it is sufficient to show that
    \begin{align}
        \label{eq:oQsbGlxhT}
        \frac{1 + \gamma\max_{i} \lambda_{\max}(\mA_i)}{\gamma} \left(\mI - \left(\gamma\mA_i + \mI\right)^{-1}\right) \succeq \mA_i.
    \end{align}
    Using the eigenvalue decomposition of $\mA_i$, the expression can be rewritten as
    \begin{align*}
        \frac{1 + \gamma\max_{i} \lambda_{\max}(\mA_i)}{\gamma} \left(\mI - \left(\gamma\mA_i + \mI\right)^{-1}\right)
        & =\frac{1 + \gamma\max_{i} \lambda_{\max}(\mA_i)}{\gamma} \left(\mI - \left(\gamma \mQ_i \blambda_i \mQ_i^\top + \mI\right)^{-1}\right) \\
        & =\frac{1 + \gamma\max_{i} \lambda_{\max}(\mA_i)}{\gamma} \mQ_i \left(\mI - \left(\gamma \blambda_i + \mI\right)^{-1}\right) \mQ_i^\top.
    \end{align*}
    Now, letting $[\mB]_{j}$ be the $j$th diagonal element of matrix $\mB$, we have
    \begin{align*}
        & \left[\frac{1 + \gamma\max_{i} \lambda_{\max}(\mA_i)}{\gamma} \left(\mI - \left(\gamma \blambda_i + \mI\right)^{-1}\right)\right]_{j}
        =\frac{1 + \gamma\max_{i} \lambda_{\max}(\mA_i)}{\gamma} \left[\mI - \left(\gamma \blambda_i + \mI\right)^{-1}\right]_{j} \\
        & =\frac{1 + \gamma\max_{i} \lambda_{\max}(\mA_i)}{\gamma} \left(1 - \frac{1}{\gamma[\blambda_i]_{j} + 1}\right)
        = \frac{1 + \gamma\max_{i} \lambda_{\max}(\mA_i)}{\gamma} \frac{\gamma \lambda_{j}(\mA_i)}{1 + \gamma \lambda_{j}(\mA_i)}
        \geq \lambda_{j}(\mA_i),
    \end{align*}
    and hence
    \begin{align*}
        \frac{\left(1 + \gamma\max_{i} \lambda_{\max}(\mA_i)\right)}{\gamma} \left(\mI - \left(\gamma\blambda_i + \mI\right)^{-1}\right) \succeq \blambda_i.
    \end{align*}
    Multiplying by $\mQ_i$ from the left and by $\mQ_i^\top$ from the right, we see that \eqref{eq:oQsbGlxhT} holds, and we can conclude that
    \begin{align*}
        (1+\gamma L_{\max}) L_{\gamma} \geq \lambda_{\max}\left(\frac{1}{n} \sum_{i=1}^n \mA_i\right) = L
    \end{align*}
    for all $\gamma \geq 0.$ Inequality \eqref{eq:pess_inequality} follows form Assumption \ref{ass:local_compl}.    

    We now turn to the second part of the theorem and show that $\pi(\gamma) \times \frac{L_{\gamma} (1 + \gamma L_{\max}) R^2}{\varepsilon} \to \pi(0) \times \frac{L R^2}{\varepsilon}$ as $\gamma \to 0$.
    The idea behind the proof is that for small enough $\gamma$, it holds that
    \begin{eqnarray*}
        (1+\gamma L_{\max}) L_{\gamma}
        &\overset{\eqref{eq:asdasfbfg}}{=}& \parens{1+\gamma \max_{i} \lambda_{\max}(\mA_i)} \frac{1}{\gamma} \lambda_{\max}\left(\frac{1}{n} \sum_{i=1}^n \left(\mI - \left(\gamma \mA_i + \mI\right)^{-1}\right)\right) \\
        &\approx& \frac{1}{\gamma} \lambda_{\max}\left(\frac{1}{n} \sum_{i=1}^n \gamma \mA_i\right) = \lambda_{\max}\left(\frac{1}{n} \sum_{i=1}^n \mA_i\right) = L.
    \end{eqnarray*}
    More formally, we have
    \begin{align*}
        \mI - \left(\mI + \gamma\mA_i\right)^{-1} = \gamma\mA_i - \gamma^2\mA_i^2 + \gamma^3\mA_i^3 + \dots = \sum_{k=1}^{\infty} (-1)^{k-1} \gamma^k\mA_i^k,
    \end{align*}
    so that
    \begin{align*}
        (1+\gamma L_{\max}) L_{\gamma} = \left(\frac{1}{\gamma}+ \max_{i} \lambda_{\max}(\mA_i)\right) \lambda_{\max}\left(\frac{1}{n} \sum_{i=1}^n \sum_{k=1}^{\infty} (-1)^{k-1} \gamma^k\mA_i^k\right).
    \end{align*}
    By applying the dominated convergence theorem to exchange the limit operation and summation, we obtain
    \begin{eqnarray*}
        \lim_{\gamma \to 0} (1+\gamma L_{\max}) L_{\gamma}
        &=& \lim_{\gamma \to 0} \parens{\left(1 + \gamma \max_{i} \lambda_{\max}(\mA_i)\right) \lambda_{\max}\left(\frac{1}{n} \sum_{i=1}^n \sum_{k=1}^{\infty} (-1)^{k-1} \gamma^{k-1} \mA_i^k\right)} \\
        &=& \lambda_{\max}\left(\frac{1}{n} \sum_{i=1}^n \sum_{k=1}^{\infty}\lim_{\gamma \to 0} (-1)^{k-1} \gamma^{k-1}\mA_i^k\right) = \lambda_{\max}\left(\frac{1}{n} \sum_{i=1}^n \mA_i\right) = L,
    \end{eqnarray*}
    as required.

    Lastly, letting $\gamma\to 0$, we have $M_{f_i}^{\gamma}(x) \to f_i(x)$ for all $x$, and consequently $M^{\gamma}(x) \to f(x)$ (see, e.g., \citep{rockafellar1998variational}). Since Algorithm \ref{algorithm:fedexprox} is equivalent to \algname{GD} for minimizing $M^{\gamma}(x)$ with stepsize $\alpha\gamma = \nicefrac{1}{L_{\gamma}}$, it follows that its iterates converge to those of vanilla \algname{GD} with stepsize~$\nicefrac{1}{L}$.
\end{proof}

\newpage

\subsection{Tighter Analysis}

\subsubsection{Full Participation Case}

\paragraph{Iteration complexity.}
We first establish the convergence rate of the algorithm.

\begin{theorem}\label{thm:fedexprox_quad_iter}
    Let Assumption \ref{ass:inter} hold and let $x_k$ be the iterates of \algname{FedExProx} (Algorithm \ref{algorithm:fedexprox}) applied to problem \eqref{eq:problem_quad}. Then
    \begin{align*}
        \norm{x_{K}-\Pi(x_K)}^{2}
        \leq \parens{1 - \alpha \gamma \parens{2 - \alpha \gamma \lambda_{\max}(\gamma)} \lambda_{\min}^+(\gamma)}^K \norm{x_0-\Pi(x_0)}^2,
    \end{align*}
    where $\lambda_{\max}(\gamma)$ and $\lambda_{\min}^+(\gamma)$ are the largest and the smallest positive eigenvalues of the matrix $\mM \eqdef \frac{1}{\gamma} \parens{\mI - \frac{1}{n}\sum_{i=1}^{n}(\gamma \mA_{i}+\mI)^{-1}}$, respectively.
    
    The optimal choice of stepsize $\gamma$ and extrapolation parameter $\alpha$ is $\alpha\gamma = \frac{1}{\lambda_{\max}(\gamma)}$, in which case the rate becomes
    \begin{align*}
        \norm{x_{K}-\Pi(x_K)}^{2}
        \leq \parens{1 - \frac{\lambda_{\min}^+(\gamma)}{\lambda_{\max}(\gamma)}}^K \norm{x_0-\Pi(x_0)}^2.
    \end{align*}
    Hence, the number of iterations needed to reach an $\epsilon$-solution is
    \begin{align*}
        K \geq \frac{\lambda_{\max}(\gamma)}{\lambda_{\min}^+(\gamma)} \log\parens{\frac{\norm{x_0 - \Pi(x_0)}^2}{\epsilon}}.
    \end{align*}
\end{theorem}

\begin{proof}[Proof of Theorem \ref{thm:fedexprox_quad_iter}]
    The update rule of \algname{FedExProx} can be written as
    \begin{align*}
        x_{k+1} & = x_{k} + \alpha \left(\frac{1}{S}\sum_{i=1}^{S}\prox_{\gamma f_i}(x_{k})-x_{k}\right) \\
        & \overset{\eqref{eq:quad_prox}}{=} x_{k} + \alpha \left(\frac{1}{S}\sum_{i=1}^{S}(\gamma \mA_i + \mI)^{-1}(x_{k}+\gamma b_i)-x_{k}\right) \\
        & = x_{k}+ \alpha\left(\frac{1}{S}\sum_{i=1}^{S}((\gamma \mA_i + \mI)^{-1}-\mI)\right)x_{k} - \alpha\left(\frac{1}{S}\sum_{i=1}^{S}((\gamma \mA_i + \mI)^{-1}-\mI)\right) \Pi(x_{k}),
    \end{align*}
    where we use the fact that $\mA_{i}x_* = b_{i}$ for all $x_*\in\cX_*$. Therefore, using symmetry of $\mA_{i}$
    \begin{align*}
        \norm{x_{k+1}-\Pi(x_{k+1})}^{2} &\leq \norm{x_{k+1}-\Pi(x_{k})}^{2} \\
        & = \norm{\left((1-\alpha)\mI + \alpha\frac{1}{n}\sum_{i=1}^{n}(\gamma \mA_{i}+\mI)^{-1}\right)(x_{k}-\Pi(x_{k}))}^{2} \\
        & = \norm{\left(\mI- \alpha \underbrace{\left(\frac{1}{n} \sum_{i=1}^{n} (\mI - (\gamma \mA_{i} + \mI)^{-1})\right)}_{\eqdef \gamma\mM}\right)(x_{k}-\Pi(x_{k}))}^{2} \\
        & = (x_{k}-\Pi(x_{k}))^{T} \parens{\mI- \alpha\gamma \mM}^{T} \parens{\mI- \alpha\gamma \mM} (x_{k}-\Pi(x_{k}))\\
        & =(x_{k}-\Pi(x_{k}))^{T}\left(\mI - 2\alpha\gamma \mM + \alpha^{2}\gamma^{2} \mM^{2}\right)(x_{k}-\Pi(x_{k})).
    \end{align*}
    Since $\mM^{2} \preceq \lambda_{\max}(\gamma) \mM$, it follows that
    \begin{align*}
        \norm{x_{k+1}-\Pi(x_{k+1})}^{2} & \leq (x_{k}-\Pi(x_{k}))^{T}\left(\mI - 2\alpha\gamma \mM + \lambda_{\max}(\gamma) \alpha^{2}\gamma^{2} \mM\right)(x_{k}-\Pi(x_{k})) \\
        & = \norm{x_{k}-\Pi(x_{k})}^{2} - \alpha\gamma(2-\lambda_{\max}(\gamma)\alpha\gamma) (x_{k}-\Pi(x_{k}))^{T} \mM (x_{k}-\Pi(x_{k})) \\
        & \leq \norm{x_{k}-\Pi(x_{k})}^{2} - \alpha\gamma(2-\lambda_{\max}(\gamma)\alpha\gamma) \lambda_{\min}^{+}(\gamma) \norm{x_{k}-\Pi(x_{k})}^{2}
    \end{align*}
    for any $\alpha\gamma$ such that $2-\lambda_{\max}(\gamma)\alpha\gamma \geq 0$, where the last inequality follows from the fact that $x_{k}-\Pi(x_{k}) \in \range(\mM)$ (see Lemma \ref{lemma:projections_SGD_q}).
    Consequently,
    \begin{align*}
        \norm{x_{k+1}-\Pi(x_{k+1})}^{2} & \leq \left(1 - \alpha\gamma(2 -\alpha\gamma \lambda_{\max}(\gamma)) \lambda_{\min}^{+}(\gamma)\right)\norm{x_{k}-\Pi(x_{k})}^{2}.
    \end{align*}
    It remains to unroll the recurrence and substitute the optimal step size $\alpha\gamma = \frac{1}{\lambda_{\max}(\gamma)}$ to obtain the rate
    \begin{align*}
        \norm{x_{k+1}-\Pi(x_{k+1})}^{2} & \leq \left(1 - \frac{\lambda_{\min}^{+}(\gamma)}{\lambda_{\max}(\gamma)}\right)\norm{x_{k}-\Pi(x_{k})}^{2}.
    \end{align*}
\end{proof}

\FEDEXPROXQITERTIGHT*

\begin{proof}
    The result follows directly from Theorem \ref{thm:fedexprox_quad_iter}, smoothness of $f$, and Fact \ref{fact:hess_m} by substituting $\lambda_{\max}(\gamma) = \lambda_{\max} (\nabla^2 M^{\gamma}) = L_{\gamma}$ and $\lambda_{\min}^+(\gamma) = \lambda_{\min}^+(\nabla^2 M^{\gamma}) = \mu^{+}_{\gamma}$.
\end{proof}

\begin{remark}
    Recall from the proof of Fact \ref{fact:m_eigen_evals} that
    \begin{align*}
        \mM = \frac{1}{n}\sum_{i=1}^{n} \mQ_i \left[\frac{\lambda_{j}(\mA_{i})}{1 + \gamma \lambda_{j}(\mA_{i})}\right]_{jj} \mQ_i^\top,
    \end{align*}
    where $\frac{\lambda_{j}(\mA_{i})}{1 + \gamma \lambda_{j}(\mA_{i})} \in \left[\frac{\lambda_{j}(\mA_{i})}{2}, \lambda_{j}(\mA_{i})\right]$ if $\gamma \leq \frac{1}{\lambda_{j}(\mA_{i})}$. Consequently, for $\gamma \leq \frac{1}{\max_{i\in[n]} \lambda_{\max}(\mA_{i})}$, we have
    \begin{align*}
        \mM = \frac{1}{n}\sum_{i=1}^{n} \mQ_i \left[\frac{\lambda_{j}(\mA_{i})}{1 + \gamma \lambda_{j}(\mA_{i})}\right]_{jj} \mQ_i^\top
        \preceq \frac{1}{n}\sum_{i=1}^{n} \mQ_i \left[\lambda_{j}(\mA_{i})\right]_{jj} \mQ_i^\top
        = \frac{1}{n}\sum_{i=1}^{n} \mA_{i}
        = \mA,
    \end{align*}
    and analogously, $\mM \succeq \frac{1}{2} \mA$. Using \eqref{eq:lmax_eig} and \eqref{eq:lmin+_eig} gives
    \begin{align*}
        \frac{1}{2} \lambda_{\max}(\mA) = \frac{1}{2} \max_{\norm{x} \leq 1}  x^\top \mA x
        \leq \lambda_{\max}(\gamma) = \max_{\norm{x} \leq 1}  x^\top \mM x
        \leq \max_{\norm{x} \leq 1}  x^\top \mA x = \lambda_{\max}(\mA),
    \end{align*}
    and similarly
    \begin{align*}
        \frac{1}{2} \lambda_{\min}^{+}(\mA) \leq \lambda_{\min}^{+}(\gamma) = \min_{\norm{x} = 1, x \in (\ker \mA)^{\perp}}  x^\top \mM x \leq \lambda_{\min}^{+}(\mA).
    \end{align*}
    Noting that $L_{\gamma} = \lambda_{\max}(\gamma)$, $\mu^{+}_{\gamma} = \lambda_{\min}^+(\gamma)$, $L = \lambda_{\max}(\mA)$ and $\mu^+ = \lambda_{\min}^{+}(\mA)$, it follows that the rates~\eqref{eq:tighter_quad} of \algname{FedExProx} and~\eqref{eq:yuveVq} of \algname{GD} coincide up to a constant factor when $\gamma$ is sufficiently small. This should come as no surprise, as Theorem~\ref{thm:pessimistic} shows that \algname{FedExProx} effectively reduces to \algname{GD} as $\gamma \to 0$.
\end{remark}

\paragraph{Time complexity.}\label{sec:fedexprox_q_time}

Let us now prove the result from Section \ref{sec:time_cond}.
We assume that each worker $i$ computes the proximal operator using an iterative method, with the running time proportional to the condition number of the subproblem. Given that $\lambda_{\min}(\mA_{i}) = 0$ for all $i\in[n]$ (otherwise the problem becomes trivial), the time required for all clients to compute $\prox_{\gamma f_i}(x_k)$ at each global iteration $k$ is $$\tau \times \left(1+\gamma\max_{i \in [n]}\lambda_{\max}(\mA_{i})\right)$$ seconds, where $\tau$ is the time per one iteration of solving the subproblem. Then, it takes $\ct$ seconds to aggregate the updates at the server and move on to the next step. Since, according to Theorem \ref{thm:fedexprox_quad_iter}, the number of iterations needed to reach an $\epsilon$-solution is
\begin{align*}
    K = \frac{\lambda_{\max}(\gamma)}{\lambda_{\min}^+(\gamma)} \log\parens{\frac{\norm{x_0 - \Pi(x_0)}^2}{\epsilon}},
\end{align*}
the total time required to solve the global problem is
\begin{align}\label{eq:fedexprox_quad_t}
    T_{\ct}(\gamma) \eqdef \frac{\lambda_{\max}(\gamma)}{\lambda_{\min}^{+}(\gamma)} \log\parens{\frac{\norm{x_0 - \Pi(x_0)}^2}{\epsilon}} \times \left(\ct + \tau \times \left(1+\gamma\max_{i \in [n]}\lambda_{\max}(\mA_{i})\right)\right).
\end{align}

\FEDEXPROXQTIME*
\begin{proof}
    Recall that $\lambda_{\max}(\gamma)$ and $\lambda_{\min}^{+}(\gamma)$ are the eigenvalues of the matrix
    \begin{align*}
        \mM = \frac{1}{n}\sum_{i=1}^{n} \frac{1}{\gamma} \parens{\mI-(\gamma \mA_{i}+\mI)^{-1}}
        \overset{\eqref{fact:m_eigen_evals}}{=} \frac{1}{n}\sum_{i=1}^{n} \mQ_i \left[\frac{\lambda_{j}(\mA_{i})}{1 + \gamma \lambda_{j}(\mA_{i})}\right]_{jj} \mQ_i^\top.
    \end{align*}
    Since if $\gamma \geq \frac{1}{\lambda_{j}(\mA_{i})}$ and $\lambda_{j}(\mA_{i}) > 0$, then $\frac{\gamma \lambda_{j}(\mA_{i})}{1 + \gamma \lambda_{j}(\mA_{i})} \in \left[\frac{1}{2}, 1\right)$, and if $\gamma \leq \frac{1}{\lambda_{j}(\mA_{i})}$, then $\frac{\gamma \lambda_{j}(\mA_{i})}{1 + \gamma \lambda_{j}(\mA_{i})} \in \left[\frac{\gamma \lambda_{j}(\mA_{i})}{2}, \gamma \lambda_{j}(\mA_{i})\right)$, it follows that
    \begin{align}
        \label{eq:yuveVq}
        \frac{\lambda_{j}(\mA_{i})}{1 + \gamma \lambda_{j}(\mA_{i})} = \begin{cases}
            \Theta\left(\frac{1}{\gamma}\right),   & \gamma \geq \frac{1}{\lambda_{j}(\mA_{i})}, \\
            \Theta\left(\lambda_{j}(\mA_{i})\right), & \gamma < \frac{1}{\lambda_{j}(\mA_{i})}.
        \end{cases}
    \end{align}
    Moreover, the identities \eqref{eq:lmax_eig} and \eqref{eq:lmin+_eig} from Fact \ref{fact:m_eigen_evals} tell us that
    \begin{align*}
        \lambda_{\max}(\gamma) = \max_{\norm{x} \leq 1} \frac{1}{n}\sum_{i=1}^{n} x^\top \mQ_i \left[\frac{\lambda_{j}(\mA_{i})}{1 + \gamma \lambda_{j}(\mA_{i})}\right]_{jj} \mQ_i^\top x
    \end{align*}
    and
    \begin{align*}
        \lambda_{\min}^{+}(\gamma) = \min_{\norm{x} = 1, x \in \ker(\mA)^{\perp}} \frac{1}{n}\sum_{i=1}^{n} x^\top \mQ_i \left[\frac{\lambda_{j}(\mA_{i})}{1 + \gamma \lambda_{j}(\mA_{i})}\right]_{jj} \mQ_i^\top x,
    \end{align*}
    where $\mA \eqdef \frac{1}{n} \sum_{i=1}^{n} \mA_i$.
    Due to \eqref{eq:lmax_eig}, \eqref{eq:lmin+_eig} and \eqref{eq:yuveVq}, the function $T_{\ct}(\gamma)$ is approximately constant for all $\gamma \leq \frac{1}{\max_{i\in[n]} \lambda_{\max}(\mA_{i})}.$ If $\gamma \geq \frac{\frac{\ct}{\tau}-1}{\max_{i\in[n]} \lambda_{\max}(\mA_{i})}$ then $\tau (1+\gamma \max_{i\in[n]} \lambda_{\max}(\mA_{i})) \geq \ct$, and consequently,
    \begin{eqnarray*}
        T_{\ct}(\gamma) &=& \frac{\lambda_{\max}(\gamma)}{\lambda_{\min}^{+}(\gamma)} \log\parens{\frac{\norm{x_0 - \Pi(x_0)}^2}{\epsilon}} \times \left(\ct + \tau \left(1+\gamma\max_{i \in [n]}\lambda_{\max}(\mA_{i})\right)\right) \\
        &\leq& 2\tau \frac{\lambda_{\max}(\gamma)}{\lambda_{\min}^{+}(\gamma)} \log\parens{\frac{\norm{x_0 - \Pi(x_0)}^2}{\epsilon}} \times \left(1+\gamma\max_{i \in [n]}\lambda_{\max}(\mA_{i})\right) \eqdef 2 \bar{g}(\gamma).
    \end{eqnarray*}
    Hence, $\bar{g}(\gamma) \leq T_{\ct}(\gamma) \leq 2 \bar{g}(\gamma)$, and the term to be minimized is
    \begin{eqnarray*}
        \bar{T}_{\ct}(\gamma) &\eqdef& \frac{\lambda_{\max}(\gamma)}{\lambda_{\min}^{+}(\gamma)} \times \left(1+\gamma\max_{i \in [n]}\lambda_{\max}(\mA_{i})\right).
    \end{eqnarray*}
    Using \eqref{eq:lmax_eig} and \eqref{eq:lmin+_eig}, this can be written as
    \begin{align*}
        \bar{T}_{\ct}(\gamma) = \frac{\max_{\norm{x} \leq 1} \frac{1}{n}\sum_{i=1}^{n} x^\top \mQ_i \left[\frac{1+\gamma\max_{i \in [n]}\lambda_{\max}(\mA_{i})}{1 + \gamma \lambda_{j}(\mA_{i})} \lambda_{j}(\mA_{i})\right]_{jj} \mQ_i^\top x}{\min_{\norm{x} = 1, x \in \ker(\mA)^{\perp}} \frac{1}{n}\sum_{i=1}^{n} x^\top \mQ_i \left[\frac{\lambda_{j}(\mA_{i})}{1 + \gamma \lambda_{j}(\mA_{i})}\right]_{jj} \mQ_i^\top x},
    \end{align*}
    where the numerator is non-decreasing and the denominator is non-increasing as a function of $\gamma$. It follows that $\bar{g}(\gamma)$ is non-decreasing for $\gamma \geq \frac{\frac{\ct}{\tau}-1}{\max_{i \in [n]}\lambda_{\max}(\mA_{i})}$. Hence, when $\frac{\ct}{\tau}\geq2$, the optimal (up to a multiplicative factor) $\gamma$ belongs to the interval
    \begin{align}\label{eq:range1}
        \left[\frac{1}{\max_{i \in [n]}\lambda_{\max}(\mA_{i})}, \frac{\frac{\ct}{\tau}-1}{\max_{i \in [n]}\lambda_{\max}(\mA_{i})}\right].
    \end{align}
    When $\frac{\ct}{\tau} < 2$, the time-complexity is non-decreasing for $\gamma \geq \frac{\frac{\ct}{\tau}-1}{\max_{i\in[n]} \lambda_{\max}(\mA_{i})}$ and approximately constant otherwise, so the choice
    \begin{align}\label{eq:range2}
        \gamma \in \left[0, \max\brac{0,\frac{\frac{\ct}{\tau}-1}{\max_{i\in[n]} \lambda_{\max}(\mA_{i})}}\right]
    \end{align}
    is optimal (again, up to a constant factor).

    Now, suppose that $\gamma \geq \frac{1}{\min_{i\in[n]}\lambda_{\min}^+(\mA_i)}$. In this case, the diagonal entries of the matrix $\left[\frac{\gamma \lambda_{j}(\mA_{i})}{1 + \gamma \lambda_{j}(\mA_{i})}\right]_{jj}$ are either to $0$ (when $\lambda_{j}(\mA_{i})=0$) or lie within the interval $\left[\frac{1}{2}, 1\right)$. Therefore, the ratio
    \begin{align*}
        \frac{\lambda_{\max}(\gamma)}{\lambda_{\min}^{+}(\gamma)}
        = \frac{\max_{\norm{x} \leq 1} \frac{1}{n}\sum_{i=1}^{n} x^\top \mQ_i \left[\frac{\gamma \lambda_{j}(\mA_{i})}{1 + \gamma \lambda_{j}(\mA_{i})}\right]_{jj} \mQ_i^\top x}{\min_{\norm{x} = 1, x \in \ker(\mA)^{\perp}} \frac{1}{n}\sum_{i=1}^{n} x^\top \mQ_i \left[\frac{\gamma\lambda_{j}(\mA_{i})}{1 + \gamma \lambda_{j}(\mA_{i})}\right]_{jj} \mQ_i^\top x}
    \end{align*}
    is approximately constant in $\gamma$. Consequently, the time
    \begin{align*}
        T_{\ct}(\gamma) \eqdef \frac{\lambda_{\max}(\gamma)}{\lambda_{\min}^{+}(\gamma)} \log\parens{\frac{\norm{x_0 - \Pi(x_0)}^2}{\epsilon}} \times \left(\ct + \tau \left(1+\gamma\max_{i \in [n]}\lambda_{\max}(\mA_{i})\right)\right)
    \end{align*}
    is an increasing function of $\gamma$ for $\gamma \geq \frac{1}{\min_{i\in[n]}\lambda_{\min}^+(\mA_i)}$, and the optimal range for $\gamma$ can be tightened from the previously established bounds \eqref{eq:range1} and \eqref{eq:range2} to
    \begin{align*}
        \left[\frac{1}{\max_{i \in [n]}\lambda_{\max}(\mA_{i})}, \min\brac{\frac{\frac{\ct}{\tau}-1}{\max_{i \in [n]}\lambda_{\max}(\mA_{i})}, \frac{1}{\min_{i\in[n]}\lambda_{\min}^+(\mA_i)}}\right]
    \end{align*}
    when $\frac{\ct}{\tau}\geq2$ and
    \begin{align*}
        \gamma \in \left[0, \max\brac{0, \min\brac{\frac{\frac{\ct}{\tau}-1}{\max_{i \in [n]}\lambda_{\max}(\mA_{i})}, \frac{1}{\min_{i\in[n]}\lambda_{\min}^+(\mA_i)}}}\right]
    \end{align*}
    when $\frac{\ct}{\tau} < 2$.
\end{proof}

\subsubsection{Partial Participation Case}

\begin{algorithm}[t]
    \caption{\algname{FedExProx} with partial participation}
    \label{algorithm:batch_fedexprox}
    \begin{algorithmic}[1]
    \STATE \textbf{Parameters:} stepsize $\gamma > 0$, extrapolation parameter $\alpha_k>0$, starting point $x_0 \in \R^d$, batch size $S$
    \FOR{$k = 0, 1, 2, \dots$}
        \STATE Sample a minibatch $\cS_k \subseteq [n]$ uniformly from all subsets of cardinality $S$
        \STATE $x_{k+1} = x_k + \alpha_k \parens{\frac{1}{S} \sum_{i\in\cS_k} \prox_{\gamma f_i} (x_k) - x_k}$
    \ENDFOR
    \end{algorithmic}
\end{algorithm}

\paragraph{Iteration complexity.}
We again start with establishing the iteration complexity of the algorithm.

\begin{theorem}\label{thm:fedexprox_quad_stoch_iter}
    Let Assumption \ref{ass:inter} hold and let $x_k$ be the iterates of minibatch \algname{FedExProx} (Algorithm~\ref{algorithm:batch_fedexprox}) applied to problem \eqref{eq:problem_quad}. Then
    \begin{align*}
        \Exp{\norm{x_{K}-\Pi(x_{K})}^{2}}
        \leq \parens{1 - \alpha \gamma \parens{2 - \alpha \gamma L_{\gamma,S}} \lambda_{\min}^+(\gamma)}^K \norm{x_0-\Pi(x_0)}^2,
    \end{align*}
    where $\lambda_{\min}^+(\gamma)$ is the smallest positive eigenvalue of the matrix $\mM \eqdef \frac{1}{\gamma} \parens{\mI - \frac{1}{n}\sum_{i=1}^{n}(\gamma \mA_{i}+\mI)^{-1}}$ and $L_{\gamma,S} \eqdef \frac{n-S}{S(n-1)} \frac{\max_{i\in[n]}\lambda_{\max}(\mA_{i})}{1+\gamma\max_{i\in[n]}\lambda_{\max}(\mA_{i})} + \frac{n(S-1)}{S(n-1)} \lambda_{\max}(\gamma)$.
    
    The optimal choice of stepsize $\gamma$ and extrapolation parameter $\alpha$ is $\alpha\gamma = \frac{1}{L_{\gamma,S}}$, in which case the rate becomes
    \begin{align*}
        \Exp{\norm{x_{K}-\Pi(x_{K})}^{2}}
        \leq \parens{1 - \frac{\lambda_{\min}^+(\gamma)}{L_{\gamma,S}}}^K \norm{x_0-\Pi(x_0)}^2.
    \end{align*}
    Hence, the number of iterations needed to reach an $\epsilon$-solution is
    \begin{align}\label{eq:fedexprox_q_iter}
        K \geq \frac{L_{\gamma,S}}{\lambda_{\min}^+(\gamma)} \log\parens{\frac{\norm{x_0 - \Pi(x_0)}^2}{\epsilon}}.
    \end{align}
\end{theorem}

\begin{remark}
    Similar to the proof of Theorem \ref{thm:quad}, the result of Theorem \ref{thm:quad_iter_stoch} follows directly from Theorem \ref{thm:fedexprox_quad_stoch_iter}, smoothness of $f$, and Fact \ref{fact:hess_m} by noting that
    $\lambda_{\max}(\gamma) = \lambda_{\max} (\nabla^2 M^{\gamma}) = L_{\gamma}$ and $\lambda_{\min}^+(\gamma) = \lambda_{\min}^+(\nabla^2 M^{\gamma}) = \mu^{+}_{\gamma}$.
\end{remark}

\begin{proof}[Proof of Theorem \ref{thm:fedexprox_quad_stoch_iter}]
    The update rule of \algname{FedExProx} can be written as
    \begin{align*}
        x_{k+1} & = x_{k} + \alpha \left(\frac{1}{S}\sum_{i\in\cS_k} \prox_{\gamma f_i}(x_{k})-x_{k}\right) \\
        & = x_{k} + \alpha \left(\frac{1}{S}\sum_{i\in\cS_k} (\gamma \mA_i + \mI)^{-1}(x_{k}+\gamma b_i)-x_{k}\right) \\
        & = x_{k}+ \alpha\left(\frac{1}{S}\sum_{i\in\cS_k} ((\gamma \mA_i + \mI)^{-1}-\mI)\right)x_{k} - \alpha\left(\frac{1}{S}\sum_{i\in\cS_k} ((\gamma \mA_i + \mI)^{-1}-\mI)\right) \Pi(x_k),
    \end{align*}
    where we use the fact that $\mA_{i}x_* = b_{i}$ for all $x_*\in\cX_*$. Therefore
    \begin{eqnarray}\label{eq:adhbewsgav}
        &&\hspace{-0.4cm}\norm{x_{k+1}-\Pi(x_{k+1})}^{2} \nonumber \\
        &\leq& \norm{x_{k+1}-\Pi(x_k)}^{2} \nonumber \\
        &=& \norm{\left((1-\alpha)\mI + \alpha\frac{1}{S}\sum_{i\in\cS_k} (\gamma \mA_i+\mI)^{-1}\right)(x_{k}-\Pi(x_k))}^{2} \nonumber \\
        &=& \norm{x_{k}-\Pi(x_k)}^2
        - 2 \alpha (x_{k}-\Pi(x_k))^{T} \parens{\mI - \frac{1}{S}\sum_{i\in\cS_k} (\gamma \mA_i+\mI)^{-1}} (x_{k}-\Pi(x_k)) \nonumber \\
        && + \alpha^2 (x_{k}-\Pi(x_k))^{T}\parens{\mI - \frac{1}{S}\sum_{i\in\cS_k} (\gamma \mA_i+\mI)^{-1}}^{T} \parens{\mI - \frac{1}{S}\sum_{i\in\cS_k} (\gamma \mA_i+\mI)^{-1}} (x_{k}-\Pi(x_k)) \nonumber \\
        &=& \norm{x_{k}-\Pi(x_k)}^2 - 2 \alpha \gamma (x_{k}-\Pi(x_k))^{T} \mM_k (x_{k}-\Pi(x_k)) \nonumber \\
        &&+ \alpha^2 \gamma^2 (x_{k}-\Pi(x_k))^{T} \mM_k^{T} \mM_k (x_{k}-\Pi(x_k)),
    \end{eqnarray}
    where $\mM_k \eqdef \frac{1}{\gamma} \parens{\mI - \frac{1}{S}\sum_{i\in\cS_k} (\gamma \mA_i+\mI)^{-1}}$.
    Next, using the fact that $\mM^2 \preceq \lambda_{\max}(\gamma) \mM$ and applying Lemma \ref{lemma:nice_m} with $\mB_i = \frac{1}{\gamma} \parens{\mI - (\gamma \mA_i+\mI)^{-1}}$, the expectation of $\mM_k^{T} \mM_k$ is
    \begin{eqnarray}\label{eq:exp_mk}
        \Exp{\mM_k^{T} \mM_k}
        &=& \frac{1}{\gamma^2} \frac{n-S}{S(n-1)} \frac{1}{n} \sum_{i=1}^n \parens{\mI - (\gamma \mA_i+\mI)^{-1}}^\top \parens{\mI - (\gamma \mA_i+\mI)^{-1}} + \frac{n(S-1)}{S(n-1)} \mM^\top \mM \nonumber \\
        &=& \frac{1}{\gamma^2} \frac{n-S}{S(n-1)} \frac{1}{n} \sum_{i=1}^n \parens{\mI - (\gamma \mA_i+\mI)^{-1}}^2 + \frac{n(S-1)}{S(n-1)} \mM^2 \nonumber \\
        &\leq& \frac{1}{\gamma} \frac{n-S}{S(n-1)} \max_{i\in[n]} \lambda_{\max}\parens{\mI - (\gamma \mA_i+\mI)^{-1}} \frac{1}{n} \sum_{i=1}^n \frac{1}{\gamma} \parens{\mI - (\gamma \mA_i+\mI)^{-1}} \nonumber \\
        &&+ \frac{n(S-1)}{S(n-1)} \lambda_{\max}(\gamma) \mM \nonumber \\
        &=& \frac{1}{\gamma} \frac{n-S}{S(n-1)} \max_{i\in[n]} \parens{\frac{\gamma \lambda_{\max}(\mA_{i})}{1+\gamma\lambda_{\max}(\mA_{i})}} \mM + \frac{n(S-1)}{S(n-1)} \lambda_{\max}(\gamma) \mM \nonumber \\
        &=& \underbrace{\parens{\frac{n-S}{S(n-1)} \frac{\max_{i\in[n]}\lambda_{\max}(\mA_{i})}{1+\gamma\max_{i\in[n]}\lambda_{\max}(\mA_{i})} + \frac{n(S-1)}{S(n-1)} \lambda_{\max}(\gamma)}}_{\eqdef L_{\gamma,S}} \mM.
    \end{eqnarray}
    Hence, taking expectation conditioned on $x_k$ in \eqref{eq:adhbewsgav} gives
    \begin{eqnarray*}
        &&\hspace{-1cm}\ExpSub{k}{\norm{x_{k+1}-\Pi(x_{k+1})}^{2}} \\
        &\leq& \norm{x_{k}-\Pi(x_k)}^2 + \alpha^2 \gamma^2 (x_{k}-\Pi(x_k))^{T} \ExpSub{k}{\mM_k^{T} \mM_k} (x_{k}-\Pi(x_k)) \\
        &&- 2 \alpha \gamma (x_{k}-\Pi(x_k))^{T} \ExpSub{k}{\mM_k} (x_{k}-\Pi(x_k)) \\
        &\overset{\eqref{eq:exp_mk}}{=}& \norm{x_{k}-\Pi(x_k)}^2 + \alpha^2 \gamma^2 L_{\gamma,S} (x_{k}-\Pi(x_k))^{T} \mM (x_{k}-\Pi(x_k)) \\
        &&- 2 \alpha \gamma (x_{k}-\Pi(x_k))^{T} \mM (x_{k}-\Pi(x_k)) \\
        &=& \norm{x_{k}-\Pi(x_k)}^2 - \alpha \gamma \parens{2 - \alpha \gamma L_{\gamma,S}} (x_{k}-\Pi(x_k))^{T} \mM (x_{k}-\Pi(x_k)) \\
        &\leq& \norm{x_{k}-\Pi(x_k)}^2 - \alpha \gamma \parens{2 - \alpha \gamma L_{\gamma,S}} \lambda_{\min}^+(\gamma) (x_{k}-\Pi(x_k))^{T} (x_{k}-\Pi(x_k)),
    \end{eqnarray*}
    where the last inequality follows from the fact that $x_{k}-\Pi(x_{k}) \in \range(\mM)$ (see Lemma \ref{lemma:projections_SGD_q}). Taking expectation again,
    \begin{eqnarray*}
        \Exp{\norm{x_{k+1}-\Pi(x_{k+1})}^{2}}
        \leq \parens{1 - \alpha \gamma \parens{2 - \alpha \gamma L_{\gamma,S}} \lambda_{\min}^+(\gamma)} \Exp{\norm{x_{k}-\Pi(x_k)}^2}.
    \end{eqnarray*}
    It remains to unroll the recurrence.
\end{proof}

\paragraph{Time complexity.}

Let us consider the same setup as in Section \ref{sec:fedexprox_q_time}, i.e., at each iteration $k$, each client computes a proximal operator $\prox_{\gamma f_i}(x_k)$ in at most $$\tau \times \left(1+\gamma\max_{i \in [n]}\lambda_{\max}(\mA_{i})\right)$$ seconds, where $\tau$ is the time per one iteration of solving the subproblem, and it takes $\ct$ seconds to aggregate the updates at the server and move on to the next step. 
According to Theorem \ref{thm:fedexprox_quad_stoch_iter}, the number of iterations needed to reach an $\epsilon$-solution is
\begin{align*}
    K = \frac{L_{\gamma,S}}{\lambda_{\min}^+(\gamma)} \log\parens{\frac{\norm{x_0 - \Pi(x_0)}^2}{\epsilon}},
\end{align*}
where $L_{\gamma,S} \eqdef \frac{n-S}{S(n-1)} \frac{\max_{i\in[n]}\lambda_{\max}(\mA_{i})}{1+\gamma\max_{i\in[n]}\lambda_{\max}(\mA_{i})} + \frac{n(S-1)}{S(n-1)} \lambda_{\max}(\gamma)$. Hence, the total time required to solve the global problem is at most
\begin{align}\label{eq:batch_fedexprox_quad_time}
    T_{\ct}(\gamma,S) \eqdef \frac{L_{\gamma,S}}{\lambda_{\min}^{+}(\gamma)} \log\parens{\frac{\norm{x_0 - \Pi(x_0)}^2}{\epsilon}} \times \parens{\ct + \tau \times \left(1+\gamma\max_{i \in [n]}\lambda_{\max}(\mA_{i})\right)}.
\end{align}

\FEDEXPROXQTIMESTOCH*
\begin{proof}
    Fix some $S\in[n]$. Then, we are interested in minimizing
    \begin{eqnarray*}
        \bar{T}_{\ct}(\gamma)
        &=& \frac{n-S}{S(n-1)} \underbrace{\frac{1}{\lambda_{\min}^{+}(\gamma)} \frac{\max_{i\in[n]}\lambda_{\max}(\mA_{i})}{1+\gamma\max_{i\in[n]}\lambda_{\max}(\mA_{i})} \parens{\ct + \tau \left(1+\gamma\max_{i \in [n]}\lambda_{\max}(\mA_{i})\right)}}_{\eqdef \bar{g}_1(\gamma)} \nonumber \\
        &&+ \frac{n(S-1)}{S(n-1)} \underbrace{\frac{\lambda_{\max}(\gamma)}{\lambda_{\min}^{+}(\gamma)} \parens{\ct + \tau \left(1+\gamma\max_{i \in [n]}\lambda_{\max}(\mA_{i})\right)}}_{\eqdef \bar{g}_2(\gamma)} \\
        &\eqdef& \frac{n-S}{S(n-1)} \bar{g}_1(\gamma) + \frac{n(S-1)}{S(n-1)} \bar{g}_2(\gamma).
    \end{eqnarray*}
    We know from Theorem \ref{thm:fedexprox_quad_time} that the second term of this expression, $\bar{g}_2(\gamma)$, is minimized by
    \begin{align*}
        \left[\frac{1}{\max_{i \in [n]}\lambda_{\max}(\mA_{i})}, \min\brac{\frac{\frac{\ct}{\tau}-1}{\max_{i \in [n]}\lambda_{\max}(\mA_{i})}, \frac{1}{\min_{i\in[n]}\lambda_{\min}^+(\mA_i)}}\right]
    \end{align*}
    when $\frac{\ct}{\tau}\geq2$ and
    \begin{align*}
        \gamma \in \left[0, \max\brac{0, \min\brac{\frac{\frac{\ct}{\tau}-1}{\max_{i \in [n]}\lambda_{\max}(\mA_{i})}, \frac{1}{\min_{i\in[n]}\lambda_{\min}^+(\mA_i)}}}\right]
    \end{align*}
    when $\frac{\ct}{\tau} < 2$.
    
    The same argument can be applied to the term $\bar{g}_1(\gamma)$. To be more precise, due to \eqref{eq:lmax_eig}, \eqref{eq:lmin+_eig}, and~\eqref{eq:yuveVq}, the function $\bar{g}_1(\gamma)$ is approximately constant for all $\gamma \leq \frac{1}{\max_{i\in[n]} \lambda_{\max}(\mA_{i})}.$
    For $\gamma > \frac{\frac{\ct}{\tau}-1}{\max_{i\in[n]} \lambda_{\max}(\mA_{i})}$, we have $\frac{\ct}{\tau} + 1+\gamma \max_{i\in[n]} \lambda_{\max}(\mA_{i}) < 2(1+\gamma \max_{i\in[n]} \lambda_{\max}(\mA_{i}))$. Consequently, letting $\bar{g}(\gamma) \eqdef 2\tau\frac{\max_{i\in[n]} \lambda_{\max}(\mA_{i})}{\lambda_{\min}^{+}(\gamma)}$, $\bar{g}_1(\gamma)$ can be bounded as
    \begin{eqnarray*}
        \frac{1}{2} \bar{g}(\gamma)
        < \bar{g}_1(\gamma)
        < \bar{g}(\gamma).
    \end{eqnarray*}
    Since $\lambda_{\min}^{+}(\gamma)$ is a non-increasing function of $\gamma$, $\bar{g}(\gamma)$ is non-decreasing. It follows that for $\gamma > \frac{\frac{\ct}{\tau}-1}{\max_{i\in[n]} \lambda_{\max}(\mA_{i})}$, $\bar{g}_1(\gamma)$ is non-decreasing in $\gamma$ (up to a constant factor), and we arrive at the same conclusions as in the case of $\bar{g}_2(\gamma)$: $\bar{g}_1(\gamma)$ is minimized by $\gamma \in \left[\frac{1}{\max_{i\in[n]} \lambda_{\max}(\mA_{i})}, \frac{\frac{\ct}{\tau}-1}{\max_{i\in[n]} \lambda_{\max}(\mA_{i})}\right]$ if $\frac{\ct}{\tau}\geq2$ and by $\gamma \in \left[0, \max\brac{0,\frac{\frac{\ct}{\tau}-1}{\max_{i\in[n]} \lambda_{\max}(\mA_{i})}}\right]$ if $\frac{\ct}{\tau}<2$.

    On the other hand, by following a similar argument as in the proof of Theorem \ref{thm:fedexprox_quad_time}, the expression
    \begin{eqnarray*}
        \frac{1}{\lambda_{\min}^{+}(\gamma)} \frac{\max_{i\in[n]}\lambda_{\max}(\mA_{i})}{1+\gamma\max_{i\in[n]}\lambda_{\max}(\mA_{i})}
        &= \frac{\frac{\gamma \max_{i\in[n]}\lambda_{\max}(\mA_{i})}{1+\gamma\max_{i\in[n]}\lambda_{\max}(\mA_{i})}}{\min_{\norm{x} = 1, x \in \ker(\mA)^{\perp}} \frac{1}{n}\sum_{i=1}^{n} x^\top \mQ_i \left[\frac{\gamma \lambda_{j}(\mA_{i})}{1 + \gamma \lambda_{j}(\mA_{i})}\right]_{jj} \mQ_i^\top x}
    \end{eqnarray*}
    is approximately constant for $\gamma \geq \frac{1}{\min_{i\in[n]}\lambda_{\min}^+(\mA_i)}$. This implies that $\bar{g}_1(\gamma)$ is increasing, and hence the intervals can be bounded above by $\frac{1}{\min_{i\in[n]}\lambda_{\min}^+(\mA_i)}$.
    
    The conclusion follows from the fact that the minimum of a convex combination of $\bar{g}_1(\gamma)$ and $\bar{g}_2(\gamma)$ must lie within the same interval.
\end{proof}

\begin{remark}
    The only term in \eqref{eq:fedexprox_q_iter} and \eqref{eq:batch_fedexprox_quad_time} that depends on $S$ is $$L_{\gamma,S} \eqdef \frac{n-S}{S(n-1)} \frac{\max_{i\in[n]}\lambda_{\max}(\mA_{i})}{1+\gamma\max_{i\in[n]}\lambda_{\max}(\mA_{i})} + \frac{n(S-1)}{S(n-1)} \lambda_{\max}(\gamma).$$
    Recall that $\lambda_{\max}(\gamma)$ is the largest eigenvalue of the matrix $\mM = \frac{1}{n}\sum_{i=1}^{n} \frac{1}{2\gamma} (\mI-(\mI+\gamma \mA_{i})^{-1})$, so
    \begin{eqnarray*}
        \lambda_{\max}(\gamma) \leq \frac{1}{\gamma} \frac{1}{n} \sum_{i=1}^{n} \frac{\lambda_{\max}(\mA_{i}) \gamma}{\lambda_{\max}(\mA_{i}) \gamma + 1}
        \leq \frac{\max_{i\in[n]} \lambda_{\max}(\mA_{i})}{\max_{i\in[n]} \lambda_{\max}(\mA_{i}) \gamma + 1}.
    \end{eqnarray*}
    As a result, since $\frac{n-S}{S(n-1)}$ is decreasing, and $\frac{n(S-1)}{S(n-1)}$ is increasing in $S$, both the iteration and time complexities are increasing functions of $S$. This underscores the advantage of involving a larger number of clients in the training process.
\end{remark}

\subsection{Lemmas}

\begin{lemma}\label{lemma:projections_SGD_q}
    Let $x_k$ be the iterates of \algname{SGD} applied to the problem $\min_x \brac{f(x) \eqdef \frac{1}{n}\sum_{i=1}^{n} \parens{\frac{1}{2} x^{T} \mB_i x + c_i^{T}x + d_i}}$, where the matrices $\mB_i$ are symmetric.
    Then $x_{k}-\Pi(x_k) \in \range(\frac{1}{n} \sum_{i=1}^{n} \mB_i)$ for all $k$.
\end{lemma}
\begin{proof}
    By definition of $\Pi$, we have
    \begin{align*}
        \Pi(x) = x - \parens{\frac{1}{n} \sum_{i=1}^{n} \mB_i}^{\dagger} \parens{\frac{1}{n} \sum_{i=1}^{n} (\mB_i x + c_i)}.
    \end{align*}
    Hence, using the identity $\mA^\top (\mA\mA^\top)^{\dagger} = \mA^{\dagger}$,
    \begin{align*}
        x_{k}-\Pi(x_{k})
         & = \parens{\frac{1}{n} \sum_{i=1}^{n} \mB_i}^{\dagger} \parens{\frac{1}{n} \sum_{i=1}^{n} (\mB_i x_k + c_i)} \\
         & = \parens{\frac{1}{n}\sum_{i=1}^{n}\mB_i}^{T} \parens{\parens{\frac{1}{n}\sum_{i=1}^{n}\mB_i} \parens{\frac{1}{n}\sum_{i=1}^{n}\mB_i}^{T}}^{\dagger} \parens{\frac{1}{n}\sum_{i=1}^{n}\mB_i \Exp{x_{k}} - c_i},
    \end{align*}
    and so, by symmetry of $\mB_{i}$, we get $x_{k}-\Pi(x_{k}) \in \range(\frac{1}{n} \sum_{i=1}^{n} \mB_i)$.
\end{proof}

\begin{lemma}\label{lemma:lmin+eigen}
    Let $\mM \eqdef \frac{1}{n}\sum_{i=1}^{n}\frac{1}{\gamma}(\mI - (\gamma \mA_{i} + \mI)^{-1})$, where $\gamma>0$ and $\mA_{i} \in \textnormal{Sym}^{d}_{+}$ for all $i\in[n]$. The smallest positive eigenvalue of $\mM$ is given by 
    \begin{align*}
        \lambda_{\min}^{+}(\gamma) \eqdef \min_{\norm{x} = 1, x \in (\ker \mA)^{\perp}} \frac{1}{n}\sum_{i=1}^{n} x^\top \mQ_i \left[\frac{\lambda_{j}(\mA_{i})}{1 + \gamma \lambda_{j}(\mA_{i})}\right]_{jj} \mQ_i^\top x,
    \end{align*}
    where $\mA=\frac{1}{n}\sum_{i=1}^{n} \mA_{i}$ and $[b_j]_{jj}$ denotes a diagonal matrix with $b_j$ as the $j$th entry. 
\end{lemma}

\begin{proof}
    First, observe that the matrices $\mI - (\gamma \mA_{i} + \mI)^{-1}$ are symmetric, and their eigenvalues are given by
    \begin{align*}
        1- \frac{1}{1+\gamma \lambda_{j}(\mA_i)} = \frac{\gamma\lambda_{j}(\mA_{i})}{1+\gamma \lambda_{j}(\mA_{i})}.
    \end{align*}
    Consequently, $\mM$ is a sum of symmetric positive definite matrices and is therefore also symmetric positive definite. We now claim that
    \begin{align*}
        \lambda_{\min}^{+}(\gamma) = \min_{\norm{x} = 1, x \in (\ker \mM)^{\perp}}  x^\top \mM x.
    \end{align*}
    First, choosing $x$ to be a multiple of the eigenvector of $\mM$ corresponding to $\lambda_{\min}^{+}(\gamma)$, we see that 
    \begin{align*}
        \lambda_{\min}^{+}(\gamma) \geq \min_{\norm{x} = 1, x \in (\ker \mM)^{\perp}} x^\top \mM x.
    \end{align*}
    To establish the reverse inequality, let $\{e_{i}\}$ be an orthonormal eigenbasis of $\mM$ and let $x$ be such that $\norm{x}=1$. Then we can write $x=\sum_{i=1}^{d} \alpha_{i}e_{i}$, and since $x\in(\ker \mM)^{\perp}$, all coefficients corresponding to an eigenvalue $0$ vanish. Thus  
    \begin{align*}
        x^\top \mM x = \sum_{i=1}^{d}\alpha_{i}^{2}\lambda_{i} \geq \lambda_{\min}^{+}(\gamma) \sum_{i=1}^{d}\alpha_{i}^{2} = \lambda_{\min}^{+}(\gamma) \norm{x}^{2} = \lambda_{\min}^{+}(\gamma).
    \end{align*}
    This proves that 
    \begin{align*}
        \lambda_{\min}^{+}(\gamma) = \min_{\norm{x} = 1, x \in (\ker \mM)^{\perp}}  x^\top \mM x.
    \end{align*}
    Now, each matrix $\mI - (\gamma \mA_{i} + \mI)^{-1}$ can be decomposed as $\mQ_i \left[\frac{\lambda_{j}(\mA_{i})}{1 + \gamma \lambda_{j}(\mA_{i})}\right]_{jj} \mQ_i^\top$. Therefore what remains to be proven is that $\ker(\mM)=\ker(\mA)$. To this end, take $x\in \ker(\mM)$. Then $\mM x=0$, and since $\gamma>0$, we have
    \begin{align}\label{eq: smallest_eigval_lemma}
        \frac{1}{n}\sum_{i=1}^{n}(\gamma \mA_{i} + \mI)^{-1}x = x.
    \end{align}
    Next, observe that 
    \begin{align*}
        \lambda_{\max}\parens{\frac{1}{n}\sum_{i=1}^{n}(\gamma \mA_{i} + \mI)^{-1}} \leq \frac{1}{n}\sum_{i=1}^{n}\lambda_{\max}\parens{(\gamma \mA_{i} + \mI)^{-1}} = \frac{1}{n}\sum_{i=1}^{n}\frac{1}{1+\gamma \lambda_{\min}(\mA_{i})}.
    \end{align*}
    Now consider two cases.
    
    First, if there exists $j\in[n]$ such that $\lambda_{\min}(\mA_{j}) > 0$, the above upper bound is strictly less than~$1$. This implies that there exists no nonzero $x$ that satisfies equation \eqref{eq: smallest_eigval_lemma}, and so $\ker(\mM)=\{0\}$. But  $\lambda_{\min}(\mA_{j}) > 0$ also implies that 
    \begin{align*}
        \lambda_{\min}\parens{\frac{1}{n} \sum_{i=1}^{n} \mA_{i}} \geq \frac{1}{n} \sum_{i=1}^{n}\lambda_{\min}(\mA_{i}) > 0.
    \end{align*}
    As a result, $\ker(\mM) = \{0\} = \ker(\mA)$.
    
    Now, let us suppose that $\lambda_{\min}(\mA_{i}) = 0$ for all $i\in[n]$. In this case, $\lambda_{\max}\parens{(\gamma \mA_{i} + \mI)^{-1}} = 1$ for all~$i\in[n]$. Since all matrices $(\gamma \mA_{i} + \mI)^{-1}$ are symmetric positive definite with maximum eigenvalue equal to $1$, it follows that equation \eqref{eq: smallest_eigval_lemma} holds if and only if 
    \begin{align*}
        (\gamma \mA_{i} + \mI)^{-1} x = x
    \end{align*} 
    for all $i\in[n]$. This is equivalent to $\gamma \mA_{i}x = 0$ for all $i\in [n]$. Since $\gamma>0$ and $\mA_{i}$ are symmetric positive definite, we can, in turn, equivalently express it as $\mA x = \frac{1}{n}\sum_{i=1}^{n}\mA_{i}x = 0$.
    
    Consequently,~\eqref{eq: smallest_eigval_lemma} holds (i.e., $x\in \ker(\mM)$) if and only if $x\in \ker(\mA)$ and hence $\ker(\mA)=\ker(\mM)$.

    The final expression in the statement of the Lemma follows from eigendecomposition (see Fact~\ref{fact:m_eigen_evals}).
\end{proof}

\newpage

\section{FedExProx under P{\L} Condition}

\begin{theorem}
    \label{thm:pl}
    Let Assumptions~\ref{ass:convex}, \ref{ass:local_lipschitz_constant}, \ref{ass:inter}, and \ref{ass:pl_condition_moreau} hold, and assume that $\alpha \gamma \leq \frac{2}{L_{\gamma, S}}$. Then the iterates of Algorithm \ref{algorithm:batch_fedexprox} satisfy
    \begin{eqnarray*}
        \Exp{\norm{x_K - \Pi(x_K)}^2}
        &\leq& \parens{1 - \alpha \gamma \parens{2 - \alpha \gamma L_{\gamma, S}} \frac{\mu^+_{\gamma}}{2}}^K \Exp{\norm{x_0 - \Pi(x_0)}^2}
    \end{eqnarray*}
    and hence
    \begin{eqnarray*}
        \Exp{f(x_K) - f(\Pi(x_K))}
        \leq \frac{L}{2} \parens{1 - \alpha \gamma \parens{2 - \alpha \gamma L_{\gamma, S}} \frac{\mu^+_{\gamma}}{2}}^K \Exp{\norm{x_0 - \Pi(x_0)}^2},
    \end{eqnarray*}
    where $L_{\gamma, S} \eqdef \parens{\frac{n-S}{S(n-1)} \frac{L_{\max}}{1+\gamma L_{\max}} + \frac{n(S-1)}{S(n-1)} L_{\gamma}}$.
    For the optimal choice of stepsize and extrapolation parameter $\alpha \gamma = \nicefrac{1}{L_{\gamma, S}}$, these rates become
    \begin{eqnarray*}
        \Exp{\norm{x_K - \Pi(x_K)}^2}
        &\leq& \parens{1 - \frac{\mu^+_{\gamma}}{2L_{\gamma, S}}}^K \Exp{\norm{x_0 - \Pi(x_0)}^2}
    \end{eqnarray*}
    and
    \begin{eqnarray*}
        \Exp{f(x_K) - f(\Pi(x_K))}
        \leq \frac{L}{2} \parens{1 - \frac{\mu^+_{\gamma}}{2L_{\gamma, S}}}^K \Exp{\norm{x_0 - \Pi(x_0)}^2}.
    \end{eqnarray*}
\end{theorem}
\begin{remark}
    The above theorem proves Theorem \ref{thm:fedexprox_pl_iter_stoch}. Furthermore, by setting $S=n$, we recover the result from Theorem \ref{thm:fedexprox_pl_iter}.
\end{remark}
\begin{proof}
    The proof closely follows the proof of Theorem $3$ from \citet{li2024power}.
    Recall from Section~\ref{sec:fedexprox_sgd} that Algorithm \ref{algorithm:batch_fedexprox} is equivalent to \algname{SGD} for minimizing $M^{\gamma}(x) \eqdef \frac{1}{n}\sum_{i=1}^{n} \nabla M_{f_i}^{\gamma}(x)$ with stepsize $\alpha\gamma$, and its updates can be written as
    \begin{align*}
        x_{k+1} = x_{k} - \alpha \gamma \frac{1}{S}\sum_{i\in\cS_k} \nabla M_{f_i}^{\gamma}(x_{k}).
    \end{align*}
    Then, for any $x_*\in\cX_*$
    \begin{eqnarray}\label{eq:pldfsdsd}
        &&\hspace{-0.5cm}\ExpSub{k}{\norm{x_{k+1} - x_*}^2} \nonumber \\
        &=& \norm{x_k - x_*}^2 - 2 \alpha \gamma \inp{x_k - x_*}{\ExpSub{k}{\frac{1}{S}\sum_{i\in\cS_k} \nabla M_{f_i}^{\gamma}(x_{k})}}
        + \alpha^2 \gamma^2 \ExpSub{k}{\norm{\frac{1}{S}\sum_{i\in\cS_k} \nabla M_{f_i}^{\gamma}(x_{k})}^2} \nonumber \\
        &\overset{\eqref{lemma:mini_equiv_global}}{=}& \norm{x_k - x_*}^2 - 2 \alpha \gamma \inp{x_k - x_*}{\nabla M^{\gamma}(x_{k}) - \nabla M^{\gamma}(x_*)}
        + \alpha^2 \gamma^2 \ExpSub{k}{\norm{\frac{1}{S}\sum_{i\in\cS_k} \nabla M_{f_i}^{\gamma}(x_{k})}^2} \nonumber \\
        &=& \norm{x_k - x_*}^2 - 2 \alpha \gamma \parens{D_{M^{\gamma}}(x_k, x_*) + D_{M^{\gamma}}(x_*, x_k)}
        + \alpha^2 \gamma^2 \ExpSub{k}{\norm{\frac{1}{S}\sum_{i\in\cS_k} \nabla M_{f_i}^{\gamma}(x_{k})}^2},
    \end{eqnarray}
    where $D_{M^{\gamma}}(x, y) \eqdef M^{\gamma}(x) - M^{\gamma}(y) - \inp{\nabla M^{\gamma} (y)}{x-y}$. Next, applying Lemma \ref{lemma:nice_m} with $\mB_i = \nabla M_{f_i}^{\gamma}(x_{k}) - \nabla M_{f_i}^{\gamma}(x_*) \in \R^{d\times 1}$, the last term in the above inequality can be written as
    \begin{eqnarray*}
        \ExpSub{k}{\norm{\frac{1}{S}\sum_{i\in\cS_k} \nabla M_{f_i}^{\gamma}(x_{k})}^2}
        &\overset{\eqref{lemma:mini_equiv_local}}{=}& \ExpSub{k}{\norm{\frac{1}{S} \sum_{i\in\cS_k} \parens{\nabla M_{f_i}^{\gamma}(x_{k}) - \nabla M_{f_i}^{\gamma}(x_*)}}^2} \\
        &\overset{\eqref{lemma:nice_m}}{=}& \frac{n-S}{S(n-1)} \frac{1}{n} \sum_{i=1}^n \norm{\nabla M_{f_i}^{\gamma}(x_{k}) - \nabla M_{f_i}^{\gamma}(x_*)}^2 \\
        &&+ \frac{n(S-1)}{S(n-1)} \norm{\frac{1}{n} \sum_{i=1}^n \parens{\nabla M_{f_i}^{\gamma}(x_{k}) - \nabla M_{f_i}^{\gamma}(x_*)}}^2.
    \end{eqnarray*}
    Looking at the first term of the inequality above, using convexity and smoothness of the functions $M_{f_i}^{\gamma}$ (Lemmas~\ref{lemma:moreau_smooth} and~\ref{lemma:moreau_convex}), we have
    \begin{eqnarray*}
        \frac{1}{n} \sum_{i=1}^n \norm{\nabla M_{f_i}^{\gamma}(x_{k}) - \nabla M_{f_i}^{\gamma}(x_*)}^2
        &\overset{\eqref{lemma:grad_breg}}{\leq}& \frac{1}{n} \sum_{i=1}^n \frac{L_i}{1+\gamma L_i} \parens{D_{M_{f_i}^{\gamma}}(x_k, x_*) + D_{M_{f_i}^{\gamma}}(x_*, x_k)} \\
        &\leq& \frac{L_{\max}}{1+\gamma L_{\max}} \frac{1}{n} \sum_{i=1}^n \parens{D_{M_{f_i}^{\gamma}}(x_k, x_*) + D_{M_{f_i}^{\gamma}}(x_*, x_k)} \\
        &=& \frac{L_{\max}}{1+\gamma L_{\max}} \parens{D_{M^{\gamma}}(x_k, x_*) + D_{M^{\gamma}}(x_*, x_k)}.
    \end{eqnarray*}
    Next, since by Lemma \ref{lemma:m_gamma_convex_smooth}, the function $M^{\gamma}$ is convex and smooth, the second term can be bounded as
    \begin{eqnarray*}
        \norm{\frac{1}{n} \sum_{i=1}^n \parens{\nabla M_{f_i}^{\gamma}(x_{k}) - \nabla M_{f_i}^{\gamma}(x_*)}}^2
        &=& \norm{\nabla M^{\gamma}(x_{k}) - \nabla M^{\gamma}(x_*)}^2 \\
        &\overset{\eqref{lemma:grad_breg}}{\leq}& L_{\gamma} \parens{D_{M^{\gamma}}(x_k, x_*) + D_{M^{\gamma}}(x_*, x_k)}.
    \end{eqnarray*}

    Applying these bounds in \eqref{eq:pldfsdsd} gives
    \begin{eqnarray*}
        &&\hspace{-0.5cm}\ExpSub{k}{\norm{x_{k+1} - x_*}^2} \\
        &\leq& \norm{x_k - x_*}^2 - 2 \alpha \gamma \parens{D_{M^{\gamma}}(x_k, x_*) + D_{M^{\gamma}}(x_*, x_k)} \\
        &&+ \alpha^2 \gamma^2 \underbrace{\parens{\frac{n-S}{S(n-1)} \frac{L_{\max}}{1+\gamma L_{\max}} + \frac{n(S-1)}{S(n-1)} L_{\gamma}}}_{\eqdef L_{\gamma, S}} \parens{D_{M^{\gamma}}(x_k, x_*) + D_{M^{\gamma}}(x_*, x_k)} \\
        &=& \norm{x_k - x_*}^2 - \alpha \gamma \parens{2 - \alpha \gamma L_{\gamma, S}} \parens{D_{M^{\gamma}}(x_k, x_*) + D_{M^{\gamma}}(x_*, x_k)}.
    \end{eqnarray*}
    By Lemma \ref{lemma:m_gamma_convex_smooth}, $M^{\gamma}$ is convex and smooth, so $D_{M^{\gamma}}(x_*, x_k) \geq 0$, and by Assumption \ref{ass:pl_condition_moreau} and Lemma \ref{lemma:star_conv}, we have
    \begin{eqnarray*}
        D_{M^{\gamma}}(x_k, x_*) = M^{\gamma}(x_k) - M^{\gamma}(x_*) \geq \frac{\mu^+_{\gamma}}{2} \norm{x_k - \Pi(x_k)}^2.
    \end{eqnarray*}
    Therefore, for $\alpha \gamma L_{\gamma, S} \leq 2$
    \begin{eqnarray*}
        \ExpSub{k}{\norm{x_{k+1} - x_*}^2}
        &\leq& \norm{x_k - x_*}^2 - \alpha \gamma \parens{2 - \alpha \gamma L_{\gamma, S}} \parens{D_{M^{\gamma}}(x_k, x_*) + D_{M^{\gamma}}(x_*, x_k)} \\
        &\leq& \norm{x_k - x_*}^2 - \alpha \gamma \parens{2 - \alpha \gamma L_{\gamma, S}} \frac{\mu^+_{\gamma}}{2} \norm{x_k - \Pi(x_k)}^2.
    \end{eqnarray*}
    Taking expectation and letting $x_* = \Pi(x_k)$ gives
    \begin{eqnarray*}
        \Exp{\norm{x_{k+1} - \Pi(x_{k+1})}^2} &\leq& \Exp{\norm{x_{k+1} - \Pi(x_k)}^2} \\
        &\leq& \Exp{\norm{x_k - \Pi(x_k)}^2} - \alpha \gamma \parens{2 - \alpha \gamma L_{\gamma, S}} \frac{\mu^+_{\gamma}}{2} \Exp{\norm{x_k - \Pi(x_k)}^2}.
    \end{eqnarray*}
    Unrolling the recurrence, we obtain the first result
    \begin{eqnarray*}
        \Exp{\norm{x_K - \Pi(x_K)}^2}
        &\leq& \parens{1 - \alpha \gamma \parens{2 - \alpha \gamma L_{\gamma, S}} \frac{\mu^+_{\gamma}}{2}}^K \Exp{\norm{x_0 - \Pi(x_0)}^2}.
    \end{eqnarray*}
    Lastly, using $L$--smoothness of $f$, it follows that
    \begin{eqnarray*}
        \Exp{f(x_K) - f(\Pi(x_K))}
        &\leq& \frac{L}{2} \Exp{\norm{x_K - \Pi(x_K)}^2} \\
        &\leq& \frac{L}{2} \parens{1 - \alpha \gamma \parens{2 - \alpha \gamma L_{\gamma, S}} \frac{\mu^+_{\gamma}}{2}}^K \Exp{\norm{x_0 - \Pi(x_0)}^2}.
    \end{eqnarray*}
    Substituting $\alpha \gamma = \nicefrac{1}{L_{\gamma, S}}$, which minimizes the expression $\alpha \gamma \parens{2 - \alpha \gamma L_{\gamma, S}}$, finishes the proof.
\end{proof}

\newpage

\section{FedExProx with Inexact Computations}\label{sec:inexact}

In practice, solving \eqref{eq:prox} exactly is often infeasible, and we can only find a vector $\prox^{\delta}_{\gamma f} (x)$ such that
\begin{align*}
    \norm{\prox^{\delta}_{\gamma f} (x) - \prox_{\gamma f} (x)}^2 \leq \delta,
\end{align*}
where $\delta$ represents the accuracy of the approximate solution to \eqref{eq:prox}. As a result, we can only calculate an inexact gradient of the Moreau envelope, defined as
\begin{align*}
    \nabla M^{\gamma,\delta}_f(x) \eqdef \frac{1}{\gamma}(x-\prox^{\delta}_{\gamma f}(x)).
\end{align*}
One can easily show that
\begin{align*}
    \norm{\nabla M_f^{\gamma,\delta}(x) - \nabla M_f^{\gamma}(x)}^2 \leq \frac{\delta}{\gamma^2}.
\end{align*}

With these inexact updates, Algorithm \ref{algorithm:fedexprox} iterates
\begin{align}\label{eq:inexact_update}
    x_{k+1} = x_{k} + \alpha_k\left(\frac{1}{n}\sum_{i=1}^{n}\prox^{\delta}_{\gamma f_i}(x_{k}) - x_{k}\right)
    = x_{k} - \alpha_k \gamma \underbrace{\frac{1}{n}\sum_{i=1}^{n} \nabla M_{f_i}^{\gamma,\delta}(x_{k})}_{\eqdef \nabla M^{\gamma,\delta}(x_{k})},
\end{align}
where
\begin{align}
    \label{eq:GjmDcYpSKqyVnbZ}
    \norm{\nabla M^{\gamma,\delta}(x) - \nabla M^{\gamma}(x)}^2
    & = \norm{\frac{1}{n}\sum_{i=1}^{n} \parens{\nabla M_{f_i}^{\gamma,\delta}(x) - \nabla M_{f_i}^{\gamma}(x)}}^2 \nonumber \\
    & = \frac{1}{\gamma^2} \norm{\frac{1}{n}\sum_{i=1}^{n} \parens{\prox^{\delta}_{\gamma f}(x)-\prox_{\gamma f}(x)}}^2 \nonumber \\
    & \leq \frac{1}{\gamma^2} \frac{1}{n}\sum_{i=1}^{n} \norm{\prox^{\delta}_{\gamma f}(x)-\prox_{\gamma f}(x)}^2
    \leq \frac{\delta}{\gamma^2}.
\end{align}

With this in mind, we proceed to our analysis.

\begin{theorem}\label{thm:fedexprox_pl_inexact}
    Consider the inexact \algname{FedExProx} method in the full participation setting defined in~\eqref{eq:inexact_update}. Let Assumptions~\ref{ass:convex}, \ref{ass:local_lipschitz_constant}, \ref{ass:inter}, and \ref{ass:pl_condition_moreau} hold, and set $\alpha \gamma = \frac{1}{2 L_{\gamma}}$. Then, the iterates of the algorithm satisfy
    \begin{eqnarray*}
        \norm{x_K - \Pi(x_K)}^2
        &\leq& \parens{1 - \frac{\mu^+_{\gamma}}{8L_{\gamma}}}^K \norm{x_0 - \Pi(x_0)}^2 + \frac{20 \delta}{\left(\mu^+_{\gamma}\right)^2 \gamma^2},
    \end{eqnarray*}
    and hence
    \begin{eqnarray}
        \label{eq:bPfek}
        f(x_K) - f(\Pi(x_K))
        \leq \frac{L}{2} \parens{1 - \frac{\mu^+_{\gamma}}{8 L_{\gamma}}}^K \norm{x_0 - \Pi(x_0)}^2 + \frac{10 L \delta}{\left(\mu^+_{\gamma}\right)^2 \gamma^2}.
    \end{eqnarray}
\end{theorem}

\begin{remark}
    Note that the price one has to pay for inexactness is minimal. To illustrate this, suppose that clients solve the local problems using \algname{GD}. Then, taking
    $$\delta = \frac{\left(\mu^+_{\gamma}\right)^2 \gamma^2 \varepsilon}{20 L},$$
    we ensure that the last term in \eqref{eq:bPfek} is small:
    \begin{align*}
        \frac{10 L \delta}{\left(\mu^+_{\gamma}\right)^2 \gamma^2} \leq \frac{\varepsilon}{2}.
    \end{align*}
    Since problem \eqref{eq:prox} is strongly convex, the local complexity is proportional to
    \begin{align*}
        \cO\left(\kappa \times \log \frac{1}{\delta}\right) = \cO\left(\kappa \times \log \left(\frac{20 L}{\left(\mu^+_{\gamma}\right)^2 \gamma^2 \varepsilon}\right)\right),
    \end{align*}
    where $\kappa \gg 1$ is the condition number. In practice, this logarithmic factor can be ignored and treated as a constant.
\end{remark}

\begin{proof}[Proof of Theorem \ref{thm:fedexprox_pl_inexact}]
    The updates can be written as 
    \begin{align*}
        x_{k+1} = x_{k} - \alpha \gamma \nabla M^{\gamma,\delta}(x_{k}).
    \end{align*}
    Thus
    \begin{align*}
        \norm{x_{k+1} - \Pi(x_{k+1})}^2 &\leq \norm{x_{k+1} - \Pi(x_{k})}^2 \\
        &= \norm{x_k - \Pi(x_{k})}^2 - 2 \alpha \gamma \inp{x_k - \Pi(x_{k})}{\nabla M^{\gamma,\delta}(x_{k})}
        + \alpha^2 \gamma^2 \norm{\nabla M^{\gamma,\delta}(x_{k})}^2 \nonumber \\
        &= \norm{x_k - \Pi(x_{k})}^2 
        - 2 \alpha \gamma \inp{x_k - \Pi(x_{k})}{\nabla M^{\gamma}(x_{k})} \\
        &\quad - 2 \alpha \gamma \inp{x_k - \Pi(x_{k})}{\nabla M^{\gamma,\delta}(x_{k}) - \nabla M^{\gamma}(x_{k})} \\
        &\quad + \alpha^2 \gamma^2 \norm{\nabla M^{\gamma}(x_{k}) + (\nabla M^{\gamma, \delta}(x_{k}) - \nabla M^{\gamma}(x_{k}))}^2.
    \end{align*}
    Using Jensen's inequality, we get
    \begin{align}\label{eq:sidgpohjnbasi}
        \norm{x_{k+1} - \Pi(x_{k+1})}^2
        &\leq \norm{x_k - \Pi(x_{k})}^2 
        - 2 \alpha \gamma \inp{x_k - \Pi(x_{k})}{\nabla M^{\gamma}(x_{k})} \nonumber \\
        &\quad - 2 \alpha \gamma \inp{x_k - \Pi(x_{k})}{\nabla M^{\gamma,\delta}(x_{k}) - \nabla M^{\gamma}(x_{k})} \nonumber \\
        &\quad + 2 \alpha^2 \gamma^2 \norm{\nabla M^{\gamma}(x_{k})}^2 + 2 \alpha^2 \gamma^2 \norm{\nabla M^{\gamma, \delta}(x_{k}) - \nabla M^{\gamma}(x_{k})}^2.
    \end{align}
    Following the same steps as in the proof of Theorem~\ref{thm:pl} gives
    \begin{align*}
        &\norm{x_k - \Pi(x_{k})}^2 
        - 2 \alpha \gamma \inp{x_k - \Pi(x_{k})}{\nabla M^{\gamma}(x_{k})} + 2 \alpha^2 \gamma^2 \norm{\nabla M^{\gamma}(x_{k})}^2 \\
        &\quad\leq \left(1 - \alpha \gamma (1 - \alpha \gamma L_{\gamma}) \mu^+_{\gamma}\right)\norm{x_k - \Pi(x_{k})}^2 \\
        &\quad\leq \left(1 - \alpha \gamma \frac{\mu^+_{\gamma}}{2}\right)\norm{x_k - \Pi(x_{k})}^2
    \end{align*}
    for $\alpha \gamma \leq \frac{1}{2 L_{\gamma}}$, and hence
    \begin{eqnarray*}
        \norm{x_{k+1} - \Pi(x_{k+1})}^2
        &\overset{\eqref{eq:sidgpohjnbasi}}{\leq}& \left(1 - \alpha \gamma \frac{\mu^+_{\gamma}}{2}\right)\norm{x_k - \Pi(x_{k})}^2 \\
        &&- 2 \alpha \gamma \inp{x_k - \Pi(x_{k})}{\nabla M^{\gamma,\delta}(x_{k}) - \nabla M^{\gamma}(x_{k})} \\
        && + 2 \alpha^2 \gamma^2 \norm{\nabla M^{\gamma, \delta}(x_{k}) - \nabla M^{\gamma}(x_{k})}^2.
    \end{eqnarray*}
    Young's inequality yields
    \begin{align*}
        \norm{x_{k+1} - \Pi(x_{k+1})}^2
        &\leq \left(1 - \alpha \gamma \frac{\mu^+_{\gamma}}{2}\right)\norm{x_k - \Pi(x_{k})}^2 
        + \alpha \gamma \lambda \norm{x_k - \Pi(x_{k})}^2 \\
        &\quad+ \frac{\alpha \gamma}{\lambda} \norm{\nabla M^{\gamma,\delta}(x_{k}) - \nabla M^{\gamma}(x_{k})}^2
        + 2 \alpha^2 \gamma^2 \norm{\nabla M^{\gamma, \delta}(x_{k}) - \nabla M^{\gamma}(x_{k})}^2,
    \end{align*}
    where $\lambda > 0$ is a free parameter. Taking $\lambda = \frac{\mu^+_{\gamma}}{4},$ we obtain
    \begin{align*}
        \norm{x_{k+1} - \Pi(x_{k+1})}^2
        &\leq \left(1 - \alpha \gamma \frac{\mu^+_{\gamma}}{4}\right)\norm{x_k - \Pi(x_{k})}^2 
        + \frac{4 \alpha \gamma}{\mu^+_{\gamma}} \norm{\nabla M^{\gamma,\delta}(x_{k}) - \nabla M^{\gamma}(x_{k})}^2 \\
        &\quad + 2 \alpha^2 \gamma^2 \norm{\nabla M^{\gamma, \delta}(x_{k}) - \nabla M^{\gamma}(x_{k})}^2 \\
        &\leq \left(1 - \alpha \gamma \frac{\mu^+_{\gamma}}{4}\right)\norm{x_k - \Pi(x_{k})}^2 
        + \frac{5 \alpha \gamma}{\mu^+_{\gamma}} \norm{\nabla M^{\gamma,\delta}(x_{k}) - \nabla M^{\gamma}(x_{k})}^2,
    \end{align*}
    where the last inequality due to $\alpha \gamma \leq \frac{1}{2 L_{\gamma}}$ and $L_{\gamma} \geq \mu^+_{\gamma}.$ Lastly, due to \eqref{eq:GjmDcYpSKqyVnbZ}, we have
    \begin{align*}
        &\norm{x_{k+1} - \Pi(x_{k+1})}^2 \leq \left(1 - \alpha \gamma \frac{\mu^+_{\gamma}}{4}\right)\norm{x_k - \Pi(x_{k})}^2 + \frac{5 \alpha \delta}{\mu^+_{\gamma} \gamma}.
    \end{align*}
    Unrolling the recursion, one can show that 
    \begin{align*}
        &\norm{x_{K} - \Pi(x_{K})}^2 \leq \left(1 - \alpha \gamma \frac{\mu^+_{\gamma}}{4}\right)^{K} \norm{x_0 - \Pi(x_{0})}^2 + \frac{20 \delta}{\left(\mu^+_{\gamma}\right)^2 \gamma^2},
    \end{align*}
    and hence
    \begin{eqnarray*}
        f(x_K) - f(\Pi(x_K))
        &\leq& \frac{L}{2} \norm{x_K - \Pi(x_K)}^2 \\
        &\leq& \frac{L}{2} \left(1 - \alpha \gamma \frac{\mu^+_{\gamma}}{4}\right)^{K} \norm{x_0 - \Pi(x_{0})}^2 + \frac{10 L \delta}{\left(\mu^+_{\gamma}\right)^2 \gamma^2}.
    \end{eqnarray*}
\end{proof}

\newpage

\section{Adaptive Extrapolation}

The optimal extrapolation parameter values in Theorems \ref{thm:quad}, \ref{thm:quad_iter_stoch}, \ref{thm:fedexprox_pl_iter} and \ref{thm:fedexprox_pl_iter_stoch} depend on $L_{\gamma}$ (or~ $L_{\gamma, S}$). This dependency can be avoided by employing \textit{adaptive} extrapolation strategies \citep{horvath2022adaptive}. Specifically, we analyze methods based on gradient diversity (\algname{FedExProx-GraDS}) and the Polyak stepsize (\algname{FedExProx-StoPS}).

\subsection{Full Participation Case}

\FEDEXPROXQITERADA*

\begin{proof}
    We start with the standard decomposition
    \begin{eqnarray}\label{eq:opiuiftrre}
        \norm{x_{k+1}-\Pi(x_{k+1})}^{2}
        &\leq& \norm{x_{k+1}-\Pi(x_k)}^{2} \nonumber \\
        &=& \norm{x_{k}-\Pi(x_k)}^{2} - 2 \alpha_{k} \gamma \inp{\frac{1}{n}\sum_{i=1}^n \nabla M^{\gamma}_{f_i}(x_{k})}{x_{k}-\Pi(x_k)} \nonumber \\
        &&+ \alpha_{k}^2 \gamma^2 \norm{\frac{1}{n}\sum_{i=1}^n \nabla M^{\gamma}_{f_i}(x_{k})}^2 \nonumber \\
         &\overset{\eqref{eq:nabla_mfi}}{=}& \norm{x_{k}-\Pi(x_k)}^{2} - 2 \alpha_{k} \gamma (x_{k} - \Pi(x_k))^\top \mM (x_{k} - \Pi(x_k)) \nonumber \\
        &&+ \alpha_{k}^2 \gamma^2 \norm{\frac{1}{n}\sum_{i=1}^n \nabla M^{\gamma}_{f_i}(x_{k})}^2.
    \end{eqnarray}
    Now, let us consider the two adaptive extrapolation strategies.
    \begin{enumerate}
        \item (\algname{FedExProx-GraDS})
        In this case, the last term of \eqref{eq:opiuiftrre} can be rewritten as
        \begin{eqnarray*}
            &&\hspace{-0.5cm}\alpha_k^2 \gamma^2 \norm{\frac{1}{n}\sum_{i=1}^n \nabla  M^{\gamma}_{f_i}(x_{k})}^{2} \\
            &=& \alpha_k \gamma^2 \frac{1}{n}\sum_{i=1}^n \norm{\nabla M^{\gamma}_{f_i}(x_{k})}^{2} \\
            &\overset{\eqref{eq:nabla_mfi}}{=}& \alpha_k \gamma^2 \frac{1}{n}\sum_{i=1}^n \norm{\frac{1}{\gamma} \parens{\mI - \parens{\gamma \mA_i + \mI}^{-1} } (x_k - \Pi(x_k))}^{2} \\
            &=& \alpha_k (x_k - \Pi(x_k))^\top \parens{\frac{1}{n}\sum_{i=1}^n \parens{\mI - \parens{\gamma \mA_i + \mI}^{-1}}^\top \parens{\mI - \parens{\gamma \mA_i + \mI}^{-1}}} (x_k - \Pi(x_k)) \\
            &\leq& \alpha_k \max_{i\in[n]}\lambda_{\max}\parens{\mI - \parens{\gamma \mA_i + \mI}^{-1}} (x_k - \Pi(x_k))^\top \parens{\frac{1}{n}\sum_{i=1}^n \parens{\mI - \parens{\gamma \mA_i + \mI}^{-1}}} (x_k - \Pi(x_k)) \\
            &=& \alpha_k \gamma \frac{\gamma \max_{i\in[n]} \lambda_{\max}(\mA_{i})}{1 + \gamma \max_{i\in[n]} \lambda_{\max}(\mA_{i})} (x_k - \Pi(x_k))^\top \mM (x_k - \Pi(x_k)).
        \end{eqnarray*}
        Substituting this in \eqref{eq:opiuiftrre} gives
        \begin{eqnarray*}
            &&\hspace{-0.5cm}\norm{x_{k+1}-\Pi(x_{k+1})}^{2} \\
            &\leq& \norm{x_{k}-\Pi(x_k)}^{2} - 2 \alpha_{k} \gamma (x_{k} - \Pi(x_k))^\top \mM (x_{k} - \Pi(x_k)) \\
            &&+ \alpha_k \gamma \frac{\gamma \max_{i\in[n]} \lambda_{\max}(\mA_{i})}{1 + \gamma \max_{i\in[n]} \lambda_{\max}(\mA_{i})} (x_k - \Pi(x_k))^\top \mM (x_k - \Pi(x_k)) \\
            &=& \norm{x_{k}-\Pi(x_k)}^{2} - \alpha_{k} \gamma \frac{2 + \gamma \max_{i\in[n]} \lambda_{\max}(\mA_{i})}{1 + \gamma \max_{i\in[n]} \lambda_{\max}(\mA_{i})} (x_k - \Pi(x_k))^\top \mM (x_k - \Pi(x_k)) \\
            &\leq& \parens{1 - \alpha_{k} \gamma \frac{2 + \gamma \max_{i\in[n]} \lambda_{\max}(\mA_{i})}{1 + \gamma \max_{i\in[n]} \lambda_{\max}(\mA_{i})} \lambda_{\min}^+(\gamma)} \norm{x_k - \Pi(x_k)}^2,
        \end{eqnarray*}
        where the last inequality follows from the fact that $x_{k}-\Pi(x_{k}) \in \range(\mM)$ (see Lemma~\ref{lemma:projections_SGD_q}).
        Applying this bound iteratively for $k=0,\ldots,K-1$ gives
        \begin{eqnarray*}
            &&\hspace{-1cm}\norm{x_K-\Pi(x_K)}^{2} \\
            &\leq& \parens{1 - \gamma \min_{k=0,\ldots,K-1} \alpha_{k} \frac{2 + \gamma \max_{i\in[n]}\lambda_{\max}(\mA_{i})}{1 + \gamma \max_{i\in[n]}\lambda_{\max}(\mA_{i})} \lambda_{\min}^+(\gamma)}^K \norm{x_0 - \Pi(x_0)}^2.
        \end{eqnarray*}

        \item (\algname{FedExProx-StoPS})
        First, note that $L_{\gamma}$--smoothness of $M^{\gamma}$ implies that
        \begin{align*}
            M^{\gamma}(x_{k}) - \inf  M^{\gamma}
            \geq \frac{1}{2 L_{\gamma}} \norm{\nabla  M^{\gamma}(x_{k})}^{2},
        \end{align*}
        and hence $\alpha_{k,S}(x_k)\geq \frac{1}{2\gamma L_{\gamma}}$.
        
        Now, the last term of \eqref{eq:opiuiftrre} is
        \begin{eqnarray*}
            \alpha_k^2 \gamma^2 \norm{\frac{1}{n}\sum_{i=1}^{n} \nabla  M^{\gamma}_{f_i}(x_{k})}^{2}
            &=& \alpha_k \gamma \frac{1}{n}\sum_{i=1}^{n} \parens{M^{\gamma}_{f_i}(x_{k})-\inf  M^{\gamma}_{f_i}} \\
            &=& \alpha_k \gamma \parens{M^{\gamma}(x_{k}) - M^{\gamma}(\Pi(x_k))} \\
            &\overset{\eqref{eq:fedexprox_m2}}{=}& \frac{1}{2} \alpha_k \gamma (x_k-\Pi(x_k))^{T} \mM (x_k-\Pi(x_k)).
        \end{eqnarray*}
        Therefore, using Lemma~\ref{lemma:projections_SGD_q}, we get
        \begin{eqnarray*}
            \norm{x_{k+1}-\Pi(x_{k+1})}^{2}
            &\overset{\eqref{eq:opiuiftrre}}{\leq}& \norm{x_{k}-\Pi(x_k)}^{2}
            - 2 \alpha_{k} \gamma (x_{k} - \Pi(x_k))^\top \mM (x_{k} - \Pi(x_k)) \\
            &&+ \frac{1}{2} \alpha_k \gamma (x - \Pi(x_k))^{T} \mM (x - \Pi(x_k)) \\
            &\leq& \parens{1 - \frac{3}{2} \alpha_{k} \gamma \lambda_{\min}^+(\gamma)} \norm{x_k - \Pi(x_k)}^2.
        \end{eqnarray*}
        Applying this bound iteratively for $k=0,\ldots,K-1$ gives
        \begin{eqnarray*}
            \norm{x_K-\Pi(x_K)}^{2}
            &\leq& \parens{1 - \frac{3}{2} \min_{k=0,\ldots,K-1} \alpha_{k} \gamma \lambda_{\min}^+(\gamma)}^K \norm{x_0 - \Pi(x_0)}^2.
        \end{eqnarray*}
    \end{enumerate}
    
\end{proof}

\subsection{Partial Participation Case}\label{sec:ada_stoch}

\begin{theorem}\label{thm:quad_iter_ada_stoch}
    Fix any $\gamma > 0$ and consider solving non-strongly convex quadratic optimization problem~\eqref{eq:main_problem}, where $f_i(x) = \frac{1}{2} x^\top \mA_i x - b_i^\top x$ for all $i \in [n],$ with $\mA_i \in \textnormal{Sym}^{d}_{+}$ and $b_i \in \R^d$. Let Assumption~\ref{ass:inter} hold and consider two adaptive extrapolation strategies: 
    \begin{enumerate}
        \item (\algname{FedExProx-GraDS}) Set
        \begin{align*}
            \alpha_k = \alpha_{k,S}^{\algname{GraDS}}(x_k, \cS_k) \eqdef \frac{\frac{1}{S}\sum_{i\in\cS^k} \norm{\nabla M^{\gamma}_{f_i}(x_{k})}^{2}}{\norm{\frac{1}{S}\sum_{i\in\cS^k} \nabla M^{\gamma}_{f_i}(x_{k})}^{2}}
            \geq 1.
        \end{align*}
        Then, the iterates of Algorithm \ref{algorithm:batch_fedexprox} satisfy
        \begin{eqnarray}\label{eq:grads_iter_stoch}
            \Exp{\norm{x_K-\Pi(x_K)}^{2}}
            &\leq& \parens{1 - \inf \alpha_{k} \gamma \frac{2 + \gamma L_{\max}}{1 + \gamma L_{\max}} \lambda_{\min}^+(\gamma)}^K \norm{x_0 - \Pi(x_0)}^2,
        \end{eqnarray}
        where
        \begin{align*}
            \inf \alpha_{k} \eqdef \inf_{x\in\R^d, \cS\subseteq[n], |\cS|=S} \alpha_{k,S}^{\algname{GraDS}}(x, \cS).
        \end{align*}
        \item (\algname{FedExProx-StoPS}) Set
        \begin{align*}
            \alpha_k = \alpha_{k,S}^{\algname{StoPS}}(x_k, \cS_k) \eqdef \frac{\frac{1}{S}\sum_{i\in\cS^k} \parens{M^{\gamma}_{f_i}(x_{k})-\inf  M^{\gamma}_{f_i}}}{\gamma\norm{\frac{1}{n}\sum_{i\in\cS^k} \nabla  M^{\gamma}_{f_i}(x_{k})}^{2}}.
        \end{align*}
        Then, the iterates of Algorithm \ref{algorithm:batch_fedexprox} satisfy
        \begin{eqnarray}\label{eq:stops_iter_stoch}
            \Exp{\norm{x_K-\Pi(x_K)}^{2}}
            &\leq& \parens{1 - \frac{3}{2} \inf \alpha_{k} \gamma \lambda_{\min}^+(\gamma)}^K \norm{x_0 - \Pi(x_0)}^2,
        \end{eqnarray}
        where
        \begin{align*}
            \inf \alpha_{k} \eqdef \inf_{x\in\R^d, \cS\subseteq[n], |\cS|=S} \alpha_{k,S}^{\algname{StoPS}}(x, \cS).
        \end{align*}
    \end{enumerate}
\end{theorem}

\begin{remark}
    In the single node setting ($S=1$),
    \begin{align*}
        \inf_{x\in\R^d, \cS\subseteq[n], |\cS|=1} \alpha_{k,1}^{\algname{GraDS}}(x, \cS) = 1.
    \end{align*}
    However, as more clients participate, the extrapolation parameter may exceed $1$, resulting in improved performance of \algname{FedExProx-GraDS}.
\end{remark}
\begin{remark}
    In the single node setting ($S=1$), Lemma \ref{lemma:moreau_smooth} implies that
    \begin{align*}
        \inf_{x\in\R^d, \cS\subseteq[n], |\cS|=1} \alpha_{k,1}^{\algname{StoPS}}(x, \cS) = \frac{1}{2} \frac{1+\gamma L_{\max}}{\gamma L_{\max}}.
    \end{align*}
    As the number of participating clients increases, the extrapolation parameter may exceed this bound, leading to better performance of \algname{FedExProx-StoPS}.
\end{remark}

\begin{proof}[Proof of Theorem \ref{thm:quad_iter_ada_stoch}]
    As in the full batch case, we start with the decomposition
    \begin{eqnarray}\label{eq:tdukctnfx}
        \norm{x_{k+1}-\Pi(x_{k+1})}^{2}
        &\leq& \norm{x_{k+1}-\Pi(x_k)}^{2} \nonumber \\
        &=& \norm{x_{k}-\Pi(x_k)}^{2} - 2 \alpha_{k} \gamma \inp{\frac{1}{S}\sum_{i\in\cS^k} \nabla M^{\gamma}_{f_i}(x_{k})}{x_{k}-\Pi(x_k)} \nonumber \\
        &&+ \alpha_{k}^2 \gamma^2 \norm{\frac{1}{S}\sum_{i\in\cS^k} \nabla M^{\gamma}_{f_i}(x_{k})}^2 \nonumber \\
         &\overset{\eqref{eq:nabla_mfi}}{=}& \norm{x_{k}-\Pi(x_k)}^{2} - 2 \alpha_{k} \gamma (x_{k} - \Pi(x_k))^\top \mM_k (x_{k} - \Pi(x_k)) \nonumber \\
        &&+ \alpha_{k}^2 \gamma^2 \norm{\frac{1}{S}\sum_{i\in\cS^k} \nabla M^{\gamma}_{f_i}(x_{k})}^2,
    \end{eqnarray}
    where $\mM_k \eqdef \frac{1}{\gamma} \parens{\mI - \frac{1}{S}\sum_{i\in\cS_k} (\gamma \mA_i+\mI)^{-1}}$.
    Now, let us consider the two adaptive extrapolation strategies.
    \begin{enumerate}
        \item (\algname{FedExProx-GraDS})
        In this case, the last term of \eqref{eq:tdukctnfx} can be rewritten as
        \begin{eqnarray*}
            &&\hspace{-0.5cm}\alpha_k^2 \gamma^2 \norm{\frac{1}{S}\sum_{i\in\cS^k} \nabla  M^{\gamma}_{f_i}(x_{k})}^{2} \\
            &=& \alpha_k \gamma^2 \frac{1}{S}\sum_{i\in\cS^k} \norm{\nabla M^{\gamma}_{f_i}(x_{k})}^{2} \\
            &\overset{\eqref{eq:nabla_mfi}}{=}& \alpha_k \gamma^2 \frac{1}{S}\sum_{i\in\cS^k} \norm{\frac{1}{\gamma} \parens{\mI - \parens{\gamma \mA_i + \mI}^{-1} } (x_k - \Pi(x_k))}^{2} \\
            &=& \alpha_k (x_k - \Pi(x_k))^\top \parens{\frac{1}{S}\sum_{i\in\cS^k} \parens{\mI - \parens{\gamma \mA_i + \mI}^{-1}}^\top \parens{\mI - \parens{\gamma \mA_i + \mI}^{-1}}} (x_k - \Pi(x_k)) \\
            &\leq& \alpha_k \max_{i\in[n]}\lambda_{\max}\parens{\mI - \parens{\gamma \mA_i + \mI}^{-1}} (x_k - \Pi(x_k))^\top \parens{\frac{1}{S}\sum_{i\in\cS^k} \parens{\mI - \parens{\gamma \mA_i + \mI}^{-1}}} (x_k - \Pi(x_k)) \\
            &=& \alpha_k \gamma \frac{\gamma \max_{i\in[n]} \lambda_{\max}(\mA_{i})}{1 + \gamma \max_{i\in[n]} \lambda_{\max}(\mA_{i})} (x_k - \Pi(x_k))^\top \mM_k (x_k - \Pi(x_k)).
        \end{eqnarray*}
        Substituting this in \eqref{eq:tdukctnfx} gives
        \begin{eqnarray*}
            &&\hspace{-0.5cm}\norm{x_{k+1}-\Pi(x_{k+1})}^{2} \\
            &\leq& \norm{x_{k}-\Pi(x_k)}^{2} - 2 \alpha_{k} \gamma (x_{k} - \Pi(x_k))^\top \mM_k (x_{k} - \Pi(x_k)) \\
            &&+ \alpha_k \gamma \frac{\gamma \max_{i\in[n]} \lambda_{\max}(\mA_{i})}{1 + \gamma \max_{i\in[n]} \lambda_{\max}(\mA_{i})} (x_k - \Pi(x_k))^\top \mM_k (x_k - \Pi(x_k)) \\
            &=& \norm{x_{k}-\Pi(x_k)}^{2} - \alpha_{k} \gamma \frac{2 + \gamma \max_{i\in[n]} \lambda_{\max}(\mA_{i})}{1 + \gamma \max_{i\in[n]} \lambda_{\max}(\mA_{i})} (x_k - \Pi(x_k))^\top \mM_k (x_k - \Pi(x_k)) \\
            &\leq& \norm{x_{k}-\Pi(x_k)}^{2} - \inf \alpha_{k} \gamma \frac{2 + \gamma \max_{i\in[n]} \lambda_{\max}(\mA_{i})}{1 + \gamma \max_{i\in[n]} \lambda_{\max}(\mA_{i})} (x_k - \Pi(x_k))^\top \mM_k (x_k - \Pi(x_k)),
        \end{eqnarray*}
        where we view $\alpha_{k}$ as a function of $x$ an $\cS$.
        Taking expectation conditioned on $x_k$, we obtain
        \begin{eqnarray*}
            &&\hspace{-4mm}\ExpSub{k}{\norm{x_{k+1}-\Pi(x_{k+1})}^{2}} \\
            &\leq& \norm{x_{k}-\Pi(x_k)}^{2} \\
            &&- \inf \alpha_{k} \gamma \frac{2 + \gamma \max_{i\in[n]} \lambda_{\max}(\mA_{i})}{1 + \gamma \max_{i\in[n]} \lambda_{\max}(\mA_{i})} (x_k - \Pi(x_k))^\top \ExpSub{k}{\mM_k} (x_k - \Pi(x_k)) \\
            &=& \norm{x_{k}-\Pi(x_k)}^{2} - \inf \alpha_{k} \gamma \frac{2 + \gamma \max_{i\in[n]} \lambda_{\max}(\mA_{i})}{1 + \gamma \max_{i\in[n]} \lambda_{\max}(\mA_{i})} (x_k - \Pi(x_k))^\top \mM (x_k - \Pi(x_k)) \\
            &\leq& \parens{1 - \inf \alpha_{k} \gamma \frac{2 + \gamma \max_{i\in[n]} \lambda_{\max}(\mA_{i})}{1 + \gamma \max_{i\in[n]} \lambda_{\max}(\mA_{i})} \lambda_{\min}^+(\gamma)} \norm{x_k - \Pi(x_k)}^2,
        \end{eqnarray*}
        where the last inequality follows from the fact that $x_{k}-\Pi(x_{k}) \in \range(\mM)$ (see Lemma~\ref{lemma:projections_SGD_q}).
        Taking full expectation and applying this bound iteratively for $k=0,\ldots,K-1$ gives the final result.

        \item (\algname{FedExProx-StoPS})
        In this case, the last term of \eqref{eq:tdukctnfx} is
        \begin{eqnarray*}
            &&\hspace{-1cm}\alpha_k^2 \gamma^2 \norm{\frac{1}{S}\sum_{i\in\cS^k} \nabla  M^{\gamma}_{f_i}(x_{k})}^{2} \\
            &=& \alpha_k \gamma \frac{1}{S}\sum_{i\in\cS^k} \parens{M^{\gamma}_{f_i}(x_{k})-\inf  M^{\gamma}_{f_i}} \\
            &\overset{\eqref{eq:fedexprox_mfi}}{=}& \frac{1}{2} \alpha_k \frac{1}{S}\sum_{i\in\cS^k} (x_k-\Pi(x_k))^{T}(\mI-(\gamma \mA_{i} + \mI)^{-1})(x_k-\Pi(x_k)) \\
            &=& \frac{1}{2} \alpha_k \gamma (x_k-\Pi(x_k))^{T} \mM_k (x_k-\Pi(x_k)).
        \end{eqnarray*}
        Therefore, \eqref{eq:tdukctnfx} gives
        \begin{eqnarray*}
            \norm{x_{k+1}-\Pi(x_{k+1})}^{2}
            &\leq& \norm{x_{k}-\Pi(x_k)}^{2} - 2 \alpha_{k} \gamma (x_{k} - \Pi(x_k))^\top \mM_k (x_{k} - \Pi(x_k)) \nonumber \\
            &&+ \frac{1}{2} \alpha_k \gamma (x_k-\Pi(x_k))^{T} \mM_k (x_k-\Pi(x_k)) \\
            % &=& \norm{x_{k}-\Pi(x_k)}^{2} - \frac{3}{2} \alpha_k \gamma (x_k-\Pi(x_k))^{T} \mM_k (x_k-\Pi(x_k)) \\
            &\leq& \norm{x_{k}-\Pi(x_k)}^{2} - \frac{3}{2} \inf \alpha_k \gamma (x_k-\Pi(x_k))^{T} \mM_k (x_k-\Pi(x_k)).
        \end{eqnarray*}
        Taking full expectation and applying Lemma~\ref{lemma:projections_SGD_q}, we obtain
        \begin{eqnarray*}
            \Exp{\norm{x_{k+1}-\Pi(x_{k+1})}^{2}}
            &\leq& \parens{1 - \frac{3}{2} \inf \alpha_{k} \gamma \lambda_{\min}^+(\gamma)} \norm{x_k - \Pi(x_k)}^2,
        \end{eqnarray*}
        and hence
        \begin{eqnarray*}
            \Exp{\norm{x_K-\Pi(x_K)}^{2}}
            &\leq& \parens{1 - \frac{3}{2} \inf \alpha_{k} \gamma \lambda_{\min}^+(\gamma)}^K \norm{x_0 - \Pi(x_0)}^2.
        \end{eqnarray*}
    \end{enumerate}
\end{proof}

\newpage

\section{Experiments}\label{sec:experiments}

\subsection{Experimental Setup}

To verify our theoretical results, we assess the empirical time complexity of \algname{FedExProx} as a function of~$\gamma$. We consider two types of optimization problems: synthetic quadratic optimization tasks and classification problems with a smooth hinge loss function. In both cases, local updates are computed using \algname{GD}. Below, we first outline the methodology common to both experiments, followed by a detailed description of the specific settings and their corresponding results.

\textbf{Total empirical time complexity.} As described in Section \ref{sec:time_cond}, we assume that one local iteration of \algname{GD} takes $\tau$ seconds. Thus, the time needed by worker $i$ to find $\prox_{\gamma f_{i}}(x_k)$ at global iteration $k$ is proportional to $\tau \times n_{ik},$ where $n_{ik}$ is the number of \algname{GD} iterations needed to find $\prox_{\gamma f_{i}}(x_{k})$ to a given fixed accuracy. In the full participation case, the total empirical time complexity is given by
\begin{align}
    \label{eq:emperical_time_compl}
    % \squeeze
    \sum_{k=0}^{K-1} \left(\ct + \tau \max_{i \in [n]} n_{ki}\right)
\end{align}
where $K$ is the number of global iterations needed for \algname{FedExProx} to converge to the desired accuracy and $\ct$ is the communication cost. In experiments involving client sampling, we consider the complexity
\begin{align}
    \label{eq:emperical_time_compl_s}
    % \squeeze
    \sum_{k=0}^{K-1} \left(\ct + \tau \max_{i \in \cS_{k}} n_{ki}\right)
\end{align}
instead, where~$\cS_{k}$ is the random subset of clients sampled at global iteration $k$.
Notice that we take the maximum over $n_{\cdot i}$ because the overall time is determined by the slowest client. In all experiments, we fix $\tau=1$ and vary $\eta$, effectively considering the ratio $\nicefrac{\eta}{\tau}$.

To ensure robust experiments, we count the number of \algname{GD} steps and communications $K$ and apply formula \eqref{eq:emperical_time_compl} (or \eqref{eq:emperical_time_compl_s}). This proxy for real wall-clock time is robust, reliable, reflects reality, and allows controlled experiments. The plots report the empirical time complexity required to find $x^k$ such that $f(x^k) - \inf_{x\in\R^d}f(x) \leq 10^{-3}$.

One of experimental tasks is to verify that this empirical complexity aligns with the theoretical complexity derived in \eqref{eq:fedexprox_quad_time_stoch} (see Section~\ref{sec:compare_theory_practice}).

\subsection{Experiments with Quadratic Optimization Problems}
\label{sec:exp_quad}

\begin{figure*}[t]
    \centering
    \includegraphics[width=0.95\textwidth]{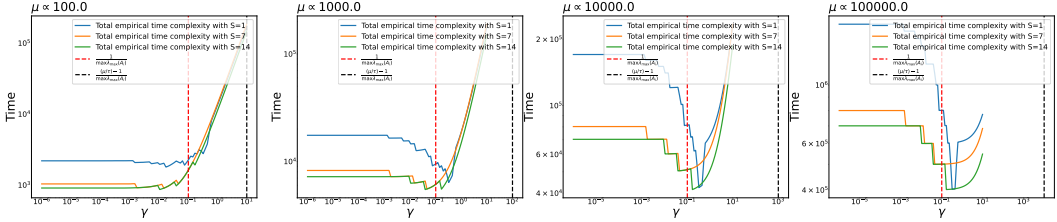}
    \caption{Empirical time complexities of \algname{FedExProx} with partial client participation on a quadratic optimization task for $S \in \{1, 7, 14\}$ clients participating in each round.}
    \label{fig:partial_quadratics}
\end{figure*}

In the first set of experiments, we consider quadratic optimization problems of the form $$f(x) = \frac{1}{n}\sum_{i=1}^{n}\frac{1}{2}x^{\top} \mA_{i}x,$$ where $\mA_{i} \in \textnormal{Sym}^{7}_{+}$, $i\in[n]$ are random positive semidefinite matrices with the smallest eigenvalue equal to zero. In particular, each worker's matrix is generated as $\mA_i = \mQ_i \mD \mQ_i^T$, where $\mQ_i$ is a random orthogonal matrix from QR decomposition of a $d\times d$ matrix with standard normal entries. The diagonal matrix $\mD$ is constructed by uniformly sampling eigenvalues between $0$ and a maximum drawn from the interval $[5,10]$.
Each worker $i$ computes the proximal mappings using \algname{GD} with a stepsize of $\nicefrac{1}{\bar{L}_i}$, where $\bar{L}_i = \lambda_{\max}(\mA_{i})+1/\gamma$ is the smoothness constant of the local subproblem~\eqref{eq:prox}.
The extrapolation parameter $\alpha$ is set to its optimal value from Theorem~\ref{thm:quad}, using an explicitly computed value of $L_{\gamma}$. The data is distributed across $n = 14$ workers.

Figures~\ref{fig:full_quadratics} and \ref{fig:partial_quadratics} display the results for the full and partial participation settings, respectively. The dashed lines indicate the theoretical bounds derived in Theorems~\ref{thm:fedexprox_quad_time} and \ref{thm:fedexprox_quad_time_stoch}, within which the optimal value of~$\gamma$ is expected to lie.
When the communication cost~$\ct$ is relatively low ($\ct \propto 100$), the optimal~$\gamma$ is close to zero. However, as~$\ct$ increases, the optimal~$\gamma$ shifts to values exceeding~$0.1$.
This results in a characteristic U-shaped curve, reflecting the nontrivial optimal choice of~$\gamma.$ These observations are fully consistent with our theoretical predictions.

\subsection{Experiments with Smooth Hinge Loss}

\subsubsection{Synthetic Data}

\begin{figure*}
    \centering
    \includegraphics[width=0.95\textwidth]{figures/hinge.pdf}
    \caption{Empirical time complexities of \algname{FedExProx} on a synthetic classification task with smooth hinge loss.}
    \label{fig:full_hinge}
\end{figure*}

\begin{figure*}
    \centering
    \includegraphics[width=0.95\textwidth]{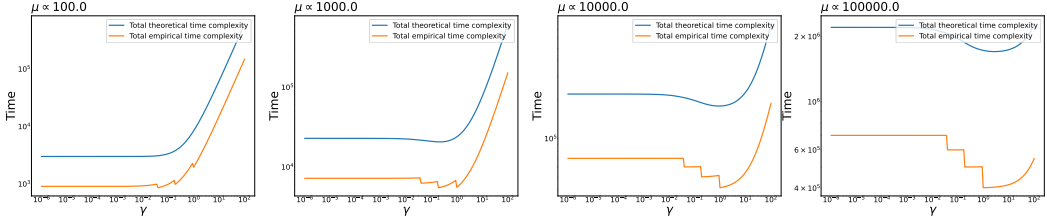}
    \caption{Comparison of theoretical time complexity \eqref{eq:fedexprox_quad_time} and empirical time complexity~\eqref{eq:emperical_time_compl}.}
    \label{fig:full_hinge_ratio}
\end{figure*}

We now turn to a classification problem involving linearly separable data $\{(x_{ij},y_{ij})\}_{i \in [n], j \in [m]}$, ensuring the interpolation regime.
Specifically, we construct a synthetic binary classification problem by independently sampling $n \times m$ data points $x_{ij}\in\R^d$ and a ground-truth weight vector $w^*\in\R^d$ from the standard normal distribution. Labels are then assigned as $y_{ij} = \textnormal{sign}(x_{ij}^T w^*)$.

The function $f :\, \R^d \rightarrow \R$ is defined as
\begin{align*}
    % \squeeze
    f(w) =\frac{1}{n}\sum\limits_{i=1}^{n}f_{i}(w) :=  \frac{1}{n}\sum\limits_{i=1}^{n}\frac{1}{m}\sum\limits_{j=1}^{m}\ell_{ij}(w),
\end{align*}
where $(x_{ij},y_{ij}) \in \R^{d} \times \{-1,1\}$ for all $i \in [n], j \in [m]$ and $\ell_{ij}$ is the smooth hinge loss, defined by 
\begin{align*}
    % \squeeze
    \ell_{ij}(w) = 
    \begin{cases} 
        0, & \text{if } y_{ij} w^T x_{ij} \geq 1, \\
        \frac{1}{2} (1 - y_{ij} w^T x_{ij})^2, & \text{if } 0 < y_{ij} w^T x_{ij} < 1, \\
        1 - y_{ij} w^T x_{ij}, & \text{if } y_{ij} w^T x_{ij} \leq 0.
    \end{cases}
\end{align*}
Each client $i \in [n]$ computes the local proximal mappings using \algname{GD} with stepsize $\nicefrac{1}{\bar{L}_{i}},$ where $\bar{L}_{i} = \frac{1}{m}\sum_{j=1}^{m} \norm{x_{ij}}^{2}$ is an upper bound on the smoothness constant of the function $f_{i}$. The extrapolation parameter $\alpha$ is determined via grid search within the theoretical range given by Lemma~\ref{lemma:m_gamma_convex_smooth}. We use $n = 4$ workers, each holding $m = 4$ datapoints, and the number of parameters is $d = 3$.

The results are presented in Figure~\ref{fig:full_hinge}. As in Section~\ref{sec:exp_quad}, these experiments again suggest a non-zero optimal choice of $\gamma$ when communication is slow. Indeed, the values of $\gamma$ minimizing the U-shapes in the last three plots are around $1$. These experiments provide further support for our theory.

\paragraph{Comparing theoretical and empirical time complexities.}
\label{sec:compare_theory_practice}

Next, we compare the theoretical and empirical time complexities  \eqref{eq:fedexprox_quad_time} and \eqref{eq:emperical_time_compl}. Both are plotted on the same scale in Figure~\ref{fig:full_hinge_ratio}. We see that our theoretical model reflects the empirical dependencies (up to a multiplicative factor).

\subsubsection{Real Data}\label{sec:arcene}

We further validate our findings on real-world data by repeating the experiment on the \texttt{ARCENE} dataset \citep{arcene_167}. We subsample $20$ random data points and restrict attention to $30$ randomly chosen features. The resulting plot in Figure~\ref{fig:full_hinge_arcene} again exhibits the characteristic U-shape behavior as the communication cost increases, mirroring the synthetic case and reinforcing the robustness of our observations

\begin{figure*}
    \centering
    \includegraphics[width=0.95\textwidth]{figures/ARCENE_results.pdf}
    \caption{Empirical time complexities of \algname{FedExProx} on the \texttt{ARCENE} classification task with smooth hinge loss.}
    \label{fig:full_hinge_arcene}
\end{figure*}

\newpage

\section{Notation}

\bgroup
\def\arraystretch{1.3}
\begin{table}[H]
\label{table:unbalanced}
	\small
	%\footnotesize
	\centering
	\begin{tabular}{|c|p{10.5cm}|}
	\hline
	\multicolumn{2}{|c|}{Notation} \\
	\hline
            $\norm{\cdot}$ & Standard Euclidean norm \\
            $\inp{\cdot}{\cdot}$ & Standard Euclidean inner product \\
            $[k]$ & $\eqdef \{1,\ldots,k\}$ \\
            $\ExpSub{k}{\cdot}$ & Expectation conditioned on the first $k$ iterations \\
            \hline
            $n$ & Number of workers/nodes/clients/devices \\
            $d$ & Dimensionality of the problem \\
            $\gamma$ & Stepsize \\
            $\alpha_k$ & Extrapolation parameter \\
            $L$, $L_i$, $L_{\max}$ & Smoothness constants (Assumption \ref{ass:local_lipschitz_constant}) \\
            $\prox_{\gamma f} (x)$ & Proximal operator of the function $f$ with parameter $\gamma$ (see \eqref{eq:prox}) \\
            $M_f^{\gamma}$ & Moreau envelope of the function $f$ with parameter $\gamma$ (see \eqref{eq:moreau}) \\
            $M^{\gamma}$ & $\eqdef \frac{1}{n} \sum_{i=1}^n M_{f_i}^{\gamma}$ \\
            $L_{\gamma}$ & Smoothness constant of $M^{\gamma}$ \\
            $\mu^{+}_{\gamma}$ & The smallest non-zero eigenvalue of the matrix $\nabla^2 M^{\gamma}$ (see Theorems \ref{thm:quad}, \ref{thm:quad_iter_stoch}) / P{\L} constant (see Theorems \ref{thm:fedexprox_pl_iter}, \ref{thm:fedexprox_pl_iter_stoch}) \\
            $\pi(\gamma)$ & Total time per global iteration of Algorithm \ref{algorithm:fedexprox} / \ref{algorithm:batch_fedexprox} \\
            $T_{\ct}(\gamma)$ & Time complexity of Algorithm \ref{algorithm:fedexprox} \\
            $T_{\ct}(\gamma,S)$ & Time complexity of Algorithm \ref{algorithm:batch_fedexprox} \\
            $\cX_*$ & $\eqdef \{x\in\R^d: \nabla f_i(x) = 0\}$ \\
            $\Pi(x)$ & Projection of $x$ onto the solution set $\cX_*$ \\
            $R^2$ & $\eqdef \norm{x_* - x_0}^2$ \\
            $D_f(x,y)$ & Bregman divergence between $x$ and $y$ associated with a function $f\,:\, \R^d \to \R$ \\
            \hline
            $\mI$ & Identity matrix \\
            $\mA_i$ & Local data matrix stored on worker $i$, with spectral decomposition $\mA_i = \mQ_i \blambda_i \mQ_i^\top$ \\
            $\mA$ & $\eqdef \frac{1}{n} \sum_{i=1}^{n} \mA_i$ \\
            $\textnormal{Sym}^{d}_{+}$ & $\eqdef \{\mX \in \R^{d\times d} \,|\, \mX=\mX^\top, \mX\succeq 0\}$ - the set of symmetric positive semidefinite matrices \\
            $\ker \mA$ & Kernel of a matrix $\mA$ \\
            $\mM_k$ & $\eqdef \frac{1}{\gamma} \parens{\mI - \frac{1}{S}\sum_{i\in\cS_k}(\gamma \mA_i+\mI)^{-1}}$ \\
            $\mM$ & $\eqdef \frac{1}{\gamma} \parens{\mI - \frac{1}{n}\sum_{i=1}^{n}(\gamma \mA_i+\mI)^{-1}}$ \\
            $\lambda_{\max}(\mB)$ & The largest eigenvalue of matrix $\mB$ \\
            $\lambda_{\min}(\mB)$ & The smallest eigenvalue of matrix $\mB$ \\
            $\lambda_{\min}^+(\mB)$ & The smallest positive eigenvalue of matrix $\mB$ \\
            $[\mB]_{j}$ & The $j$th diagonal element of a diagonal matrix $\mB$ \\
            $[b_j]_{jj}$ & A diagonal matrix with $b_j$ as the $j$th entry \\
        \hline
	\end{tabular}
	\caption{Frequently used notation.}
	\label{table:notation}
\end{table}
\egroup

\paragraph{Note on LLM Usage.}
Large Language Models were used to assist in polishing the writing of the manuscript. LLM assistance did not contribute to the scientific content of the paper.

\end{document}